\documentclass[onefignum,onetabnum,b5paper]{siamart190516}
\pdfoutput=1

\headers{Occupation measure bounds for variational problems}{G. Fantuzzi and I. Tobasco}

\title{Sharpness and non-sharpness of occupation measure bounds for integral variational problems\thanks{
		\today.\funding{G.F.\ was supported by an Imperial College Research Fellowship. I.T.\ was supported by National Science Foundation Award DMS-2025000.}}}

\author{Giovanni Fantuzzi\thanks{Mathematics, Friedrich-Alexander-Universit\"at Erlangen--N\"urnberg (\email{giovanni.fantuzzi@fau.de})}
  \and Ian Tobasco\thanks{Mathematics, Statistics \& Computer Science, University of Illinois at Chicago (\email{itobasco@uic.edu})}
}

\usepackage{amsopn}
\usepackage{amsfonts}
\usepackage{amssymb}
\usepackage{mathrsfs}
\usepackage{euscript}
\usepackage{tikz}
\usepackage{mathtools}
\usepackage{graphicx}
\usepackage{epstopdf}
\usepackage{algorithmic}
\usepackage{esint}

\usepackage[shortlabels]{enumitem}

\crefname{subsection}{section}{sections}
\Crefname{subsection}{Section}{Sections}

\ifpdf
\DeclareGraphicsExtensions{.eps,.pdf,.png,.jpg}
\else
\DeclareGraphicsExtensions{.eps}
\fi

\newsiamremark{remark}{Remark}
\newsiamremark{hypothesis}{Hypothesis}
\newsiamremark{example}{Example}
\crefname{hypothesis}{Hypothesis}{Hypotheses}
\newsiamthm{claim}{Claim}
\newsiamthm{assumption}{Assumption}

\AtBeginDocument{%
	\setlength{\oddsidemargin}{\dimexpr(\paperwidth-\textwidth)/2-1in}%
	\setlength{\evensidemargin}{\oddsidemargin}%
	\setlength{\topmargin}{%
		\dimexpr(\paperheight-\textheight)/2-\headheight-\headsep-1in}%
}

\DeclareMathOperator{\rank}{rank}

\DeclareMathOperator{\supp}{supp}

\let\inf\relax
\DeclareMathOperator*\inf{\vphantom{p}inf}

\renewcommand{\epsilon}{\varepsilon}
\renewcommand{\phi}{\varphi}

\newcommand{\R}{\mathbb{R}}

\newcommand{\D}{\mathcal{D}}
\newcommand{\dVolume}{\,{\rm d} x}
\newcommand{\dSurf}{\,{\rm d}S}
\newcommand{\LagrangeFunction}{\mathscr{L}}

\newcommand{\abs}[1]{\left\vert #1 \right\vert}
\newcommand{\weakstar}{weak-$\ast$}

\newcommand{\optimalValue}{\mathcal{F}^*}
\newcommand{\occBound}{\mathcal{F}_{\rm omr}}
\newcommand{\pdrBound}{\mathcal{L}_{\rm pdr}}
\newcommand{\wkstarto}{\overset{\ast}{\rightharpoonup}}

\newcommand{\measSpace}{\mathbb{M}^p}
\newcommand{\phiSpace}[1]{\Phi^{#1}}
\newcommand{\GYMspace}[1]{{\rm GY}^{#1}}
\newcommand{\gym}{\lambda_x}
\newcommand{\dgym}{{\rm d}\gym}
\newcommand{\omSpace}[1]{\mathbb{O}^{#1}}

\makeatletter
\newcommand{\subalign}[1]{%
	\vcenter{%
		\Let@ \restore@math@cr \default@tag
		\baselineskip\fontdimen10 \scriptfont\tw@
		\advance\baselineskip\fontdimen12 \scriptfont\tw@
		\lineskip\thr@@\fontdimen8 \scriptfont\thr@@
		\lineskiplimit\lineskip
		\ialign{\hfil$\m@th\scriptstyle##$&$\m@th\scriptstyle{}##$\crcr
			#1\crcr
		}%
	}
}
\newcommand*\fsize{\dimexpr\f@size pt\relax}
\makeatother

\ifpdf
\hypersetup{
  pdftitle={Sharpness and non-sharpness of occupation measure bounds for integral variational problems},
  pdfauthor={G. Fantuzzi, I. Tobasco},
}
\fi

\begin{document}

\maketitle

\begin{abstract}
	We analyze two recently proposed methods to establish \emph{a priori} lower bounds on the minimum of general integral variational problems. The methods, which involve either `occupation measures' or a `pointwise dual relaxation' procedure, are shown to produce the same lower bound under a coercivity hypothesis ensuring their strong duality.
    We then show by a minimax argument that the methods actually  evaluate the  minimum for classes of one-dimensional, scalar-valued, or convex multidimensional problems. For generic problems, however, these methods should fail to capture the minimum and produce non-sharp lower bounds.
    We demonstrate this using two examples, the first of which is one-dimensional and scalar-valued with a non-convex constraint, and the second of which is multidimensional and non-convex in a different way. The latter example emphasizes the existence in multiple dimensions of nonlinear constraints on gradient fields that are ignored by occupation measures, but are built into the  finer theory of gradient Young measures. 
\end{abstract}



\section{Introduction}\label{s:intro}
Integral variational problems arise throughout mathematics and the applied sciences. A typical problem takes the form
\begin{equation}\label{e:basic-functional}
\inf_{\substack{u \in W^{1,p}(\Omega; \R^m) \\ +\text{constraints}}} \int_{\Omega} f(x,u,\nabla u) \, \dVolume,
\end{equation}
where the minimization is over the space of Sobolev maps $W^{1,p}(\Omega; \R^m)$ consisting of all $u : \Omega\subset\R^n \to \R^m$ with $p$-integrable dervatives $(\nabla u)_{ij} = \smash{\frac{\partial u_i}{\partial x_j}}$ for $i=1,\dots,m$, $j=1,\dots,n$.
It is generally difficult to find the minimum value $\optimalValue$ of \cref{e:basic-functional}. Traditional optimization algorithms based on gradient or Newton methods compute stationary points, which need not be global minimizers when the problem is non-convex. Put another way, knowledge of a stationary point provides at best an upper bound $U\geq\optimalValue$ in the general case. To complete the picture, one desires methods for proving \emph{a priori} lower bounds $L\leq\optimalValue$, whose values can  be compared with upper bounds to estimate the minimum from above and below. Finding a lower bound is non-trivial, since it requires addressing general properties of admissible maps. Standard approaches include the `translation method' from the theory of composite materials \cite{Firoozye1991,Milton2002composites}, or the `calibration method', which has been used in geometry~\cite{Buttazzo1998,DePhilippis2009}. 

This paper discusses a recent method for finding lower bounds on integral variational problems based on the use of `occupation measures'~\cite{Korda2018,Korda2022} (see also \cite{Awi2014}). 
The basic idea is to replace the original problem \cref{e:basic-functional} with a convex minimization problem posed over measures obtained via a push-forward operation. The resulting `occupation measure relaxation' bound $\occBound$ is never larger than the original minimum $\optimalValue$.
A natural question is whether $\occBound=\optimalValue$; the answer, as we shall show, depends on the nature of the problem. 
As we were preparing this article for submission, we became aware of related progress in \cite{Korda2022,hkkz2023arxiv}. Our results are the same in some cases, but differ in general. Specific comparisons to help the reader sort through what is known about the sharpness/non-sharpness of the occupation measures method appear later in this introduction and throughout the article.

The idea of relaxing a variational problem posed over maps into one involving other quantities is much older than the notion of occupation measures, and dates back at least to Young's generalized curves~\cite{Young1980}. In the modern calculus of variations, the notions of Young measures and gradient Young measures~\cite{Pedregal1997,Pedregal1999,Rindler2018} are nowadays used to find measure-theoretic `relaxed problems' that are well-posed and whose optimal values are guaranteed to capture the original minimum, $\optimalValue$.
Related to this is the notion of quasiconvex functions (in the sense of Morrey~\cite{Morrey1952}), which are the natural class of integrands for which minimizers are gauranteed to exist~\cite{Dacorogna2008,Morrey1952,Rindler2018}. It is well known that to find a sharp lower bound on a general integral functional one can replace the integrand with its quasiconvex hull. Measure-theoretically, this is the same as extending the minimization from Sobolev maps to couples of maps and their compatible gradient Young measures.
Although these measures have been characterized through Jensen-type inequalities involving quasiconvex (in addition to convex) integrands~\cite{Kinderlehrer1994}, the lack of a `local' test for quasiconvexity~\cite{Kristensen1999} complicates applications of this result. Notable exceptions include one-dimensional ($n=1$) and scalar-valued ($m=1$) problems, for which quasiconvexity reduces to convexity. In these cases, gradient Young measures have been explored numerically as a way to approximate global minima~\cite{Aranda2011,Bartels2006,Egozcue2003,Jaramillo2021,Meziat2005,Meziat2008,Meziat2010,Meziat2006}.

Occupation measures are instead defined via linear constraints regardless of the dimensions $n$ and $m$ of the problem.  This has been motivated~\cite{Korda2018} by the possibility of using semidefinite programming to compute the optimal lower bound $\occBound\leq \optimalValue$ provable with occupation measures, for polynomial integrands and constraints. However, this reliance on linear constraints should mean that $\occBound\neq \optimalValue$ for generic non-convex problems, and especially in the vectorial case where $n,m\geq 2$. Our key point is that, without extra hypotheses, the set of occupation measures for a given problem (`relaxed occupation measures' in the language of~\cite{Korda2022}) should be strictly larger than the set of gradient Young measures. Since the latter give sharp lower bounds, the former should not, except in particular cases. We illustrate this with a well-known example from the theory of relaxed variational problems, involving a non-convex `double-well' potential $f(\nabla u)$ built with pairs of incompatible matrices (see  \cref{s:failure}). A separate example of non-sharpness appeared in~\cite{Korda2022}, which, in contrast to ours, uses an integrand $f(x,u,\nabla u)$ that is convex in $\nabla u$ and non-convex in $u$.

Due to the examples, structural hypotheses are needed for the occupation measure relaxation of a variational problem to be sharp. Given known results on the sharpness of gradient Young measure relaxations~\cite{Pedregal1997,Pedregal1999,Rindler2018}, we expect  occupation measures to produce sharp lower bounds ($\occBound=\optimalValue$) if the integrand $f(x,u,\nabla u)$ is convex in $\nabla u$ and if $n$ or $m$ is equal to one. Surprisingly, this remains an open problem.
Here, using the Legendre transform, we prove that $\occBound=\optimalValue$ for integrands  $f(x,u,\nabla u) = f_0(x,\nabla u) + f_1(x,u)$ where $f_0$ and $f_1$ are convex in $\nabla u$ and $u$, and $n$ and $m$ are arbitrary. After this paper was submitted, such sharpness was proved for $f(x,u,\nabla u)$ that are jointly convex in $u$ and $\nabla u$, and for arbitrary $n,m$ \cite{hkkz2023arxiv}. (Convexity is not necessary for sharpness, however, as occupation measures produce sharp lower bounds for some non-convex problems~\cite{Chernyavskiy2020}.)

Finally, we address a general drawback of the occupation measures approach: it requires the solution of an infinite-dimensional linear program to produce even a single bound (sharp or not). The `pointwise dual' approach from~\cite{Chernyavskiy2020}, instead, leads to a linear program over continuous functions where admissible functions give lower bounds. In \cref{s:duality}, we show that these approaches are dual, i.e., they are the min-max and max-min versions of a saddle-point problem; in particular, the occupation measures bound $\occBound$ is greater than or equal to the best pointwise dual lower bound $\pdrBound$. Under a suitable coercivity hypothesis for the original minimization problem guaranteeing its optimal value is finite, we prove the strong duality $\pdrBound = \occBound$.

The rest of this paper is organized as follows. \Cref{s:setup} states the general class of integral minimization problems we study. \Cref{s:om-relaxation} reviews the occupation measures approach to proving lower bounds, following~\cite{Korda2018}. \Cref{s:duality} derives the alternative pointwise dual relaxation from~\cite{Chernyavskiy2020} via a minimax argument, and proves the weak and strong duality results $\occBound\geq\pdrBound$ and $\occBound=\pdrBound$ (the former always holds, the latter holds under a coercivity hypothesis). \Cref{s:convex-integrands} proves the sharpness result $\optimalValue=\occBound=\pdrBound$ for various convex or convexifiable problems by exhibiting an optimal pointwise dual bound.
In general, when sharpness holds, one can find optimality conditions differing from the usual Euler--Lagrange ones; we explain this briefly in \cref{s:optimizers}.
\Cref{s:failure} presents our two counterexamples in which $\optimalValue>\occBound$, and includes a discussion of the relationship between occupation measures and gradient Young measures. \Cref{s:conclusions} concludes with some final remarks.

\section{Setup}\label{s:setup}
This paper studies general integral variational problems of the form
\begin{equation}
\label{e:inf-problem-general}
\optimalValue := \inf_{
	\substack{u \in W^{1,p}(\Omega; \R^m)\\
		\text{\cref{e:constraints-general-1}--\cref{e:constraints-general-3}}}
} \;
\int_{\Omega} f(x,u,\nabla u) \, \dVolume
+ \int_{\partial\Omega} g(x,u) \dSurf.
\end{equation}
Here, $\Omega \subset \R^n$ is an open bounded Lipschitz domain with boundary $\partial\Omega$, and $\dVolume$ and $\dSurf$ are the usual volume and $(n-1)$-dimensional surface measures. The minimization is over weakly differentiable functions $u:\Omega \to \R^m$ in the Sobolev space $W^{1,p}(\Omega;\R^m)$ with some fixed $p \in (1,\infty)$, subject to the constraints
\begin{subequations}
    \label{e:constraints}
	\begin{gather}
	\label{e:constraints-general-1}
	\int_{\Omega} a(x,u,\nabla u) \dVolume +
	\int_{\partial\Omega} b(x,u) \dSurf = 0,\\
	\label{e:constraints-general-2}
	c(x,u,\nabla u) = 0 \quad\text{ a.e. on } \Omega,\\
	\label{e:constraints-general-3}
	d(x,u) = 0 \quad\text{ a.e. on }\partial \Omega.
	\end{gather}
\end{subequations}
The functions $f,a,c:\Omega \times \mathbb{R}^m \times \mathbb{R}^{m \times n} \to \mathbb{R}$ and $g,b,d:\partial\Omega \times \mathbb{R}^m \to \mathbb{R}$ are assumed to be continuous, but not necessarily bounded.
Moreover, $f$, $g$, $a$ and $b$ are assumed to grow no faster than a degree-$p$ polynomial in the second and third arguments, e.g.
\begin{equation}\label{e:p-growth}
	\abs{f(x,u,F)} \lesssim 1 + |u|^p + |F|^p \quad \text{a.e. } x \in\Omega.
\end{equation}
This last restriction ensures that all integrals above are well-defined.

\section{Occupation and boundary measures}\label{s:om-relaxation}

As explained in~\cite{Korda2018}, lower bounds on the optimal value of problem~\cref{e:inf-problem-general} can be derived by first posing the minimization over so-called \textit{occupation} and \textit{boundary measures} generated by admissible functions $u$, and then by convexifying this measure-theoretic problem into a linear program over a larger set of measures.
This section reviews the derivation of this linear program to set the stage for the remainder of the article.

\subsection{Measures generated by Sobolev functions}\label{ss:om-basics}
The \emph{occupation measure} $\mu$ generated by $u \in W^{1,p}(\Omega; \R^m)$ is the pushforward of the Lebesgue measure on $\Omega$ by the map $x \mapsto (x,u(x),\nabla u(x))$ from $\Omega$ into $\Omega \times \R^m \times \R^{m\times n}$. Likewise, the \emph{boundary measure} $\nu$ generated by $u$ is the pushforward of the surface measure on $\partial\Omega$ by the map $x \mapsto (x,u(x))$ from $\partial\Omega$ into $\partial\Omega \times \R^m$. Precisely,
\begin{subequations}
	\begin{align}
		\label{e:mu-integral}
		\int_{\Omega} h(x,u(x),\nabla u(x)) \,\dVolume
		&= \int_{\Omega \times \R^m \times \R^{m \times n}} h(x,y,z) \,{\rm d}\mu(x,y,z)
		\\[1ex]
		\int_{\partial\Omega} \ell(x,u(x)) \,\dSurf
		&= \int_{\partial\Omega \times \R^m} \ell(x,y) \,{\rm d}\nu(x,y)
		\label{e:nu-integral}
	\end{align}
\end{subequations}
for all continuous functions $h \in C(\overline{\Omega} \times \R^m \times \R^{m \times n})$ and $\ell \in C(\partial\Omega \times \R^m)$ satisfying $p$-growth bounds analogous to \cref{e:p-growth}. To lighten the notation we denote integration against measures using angled brackets:
\begin{subequations}
	\begin{align}
	\langle h, \mu \rangle &:= \int_{\Omega \times \R^m \times \R^{m \times n}} h(x,y,z) \,{\rm d}\mu(x,y,z)\\
	\langle \ell, \nu \rangle &:= \int_{\partial\Omega \times \R^m} \ell(x,y) \,{\rm d}\nu(x,y)
	\end{align}
\end{subequations}

Occupation and boundary measures generated by Sobolev functions subject to the constraints~\cref{e:constraints} of the variational problem~\cref{e:inf-problem-general} satisfy a number of conditions, such as moment bounds. We summarize these conditions now. Define the sets
\begin{subequations}
\begin{align}
    \label{e:Gamma-def}
    \Gamma &:= \{(x,y,z) \in \overline{\Omega} \times \R^m \times \R^{m\times n}: \,c(x,y,z) = 0 \},\\
    \Lambda &:=\{(x,y) \in \partial\Omega \times \R^m: \,d(x,y) = 0 \}.
    \label{e:Lambda-def}
\end{align}
\end{subequations}
We use standard multi-index notation, such as $x^\alpha = x_1^{\alpha_1} \cdots x_n^{\alpha_n}$ for a multivariate monomial and $\abs{\alpha} = \alpha_1+ \cdots + \alpha_n$ for the order of the multi-index $\alpha$.
\begin{lemma}\label{th:simple-constraints}
Let $\mu$ and $\nu$ be the occupation and boundary measures generated by a function $u \in W^{1,p}(\Omega;\R^m)$ that satisfies the constraints~\cref{e:constraints}. Then, $\mu$ and $\nu$ have bounded moments of order $p$ or less, i.e.,
\begin{subequations}
	\label{e:p-moment-bounds}
	\begin{gather}
		\langle x^\alpha y^\beta z^\gamma,\, \mu \rangle
		< \infty \quad \forall (\alpha,\beta,\gamma) \in \mathbb{N}^n \times \mathbb{N}^m \times \mathbb{N}^{m\times n} \text{ s.t. } |\beta| + |\gamma| \leq p,
		\\
		\langle x^\alpha y^\beta,\, \nu \rangle
		< \infty \quad \forall (\alpha,\beta) \in \mathbb{N}^n \times \mathbb{N}^m \text{ s.t. } |\beta| \leq p.
	\end{gather}
\end{subequations}
Moreover,
\begin{subequations}
\begin{gather}
    \label{e:mu-nu-equalities}
    \langle a, \mu \rangle + \langle b, \nu \rangle = 0,\\
    \label{e:marginal-constraints-mu}
    \textstyle\langle h, \mu\rangle = \int_{\Omega} h(x)  \dVolume \quad \forall h \in C(\overline{\Omega}),\\
    \label{e:marginal-constraints-nu}
    \textstyle\langle \ell, \nu\rangle = \int_{\partial\Omega} \ell(x)  \dSurf \quad \forall \ell \in C(\partial\Omega),\\
	\label{e:support-constraints-mu}
    \supp(\mu) \subseteq \Gamma,\\
    \label{e:support-constraints-nu}
    \supp(\nu) \subseteq \Lambda.
\end{gather}
\end{subequations}
\end{lemma}

\begin{proof}
	The boundedness of all moments of order $p$ or less follows from the definition of $\mu$ and $\nu$ as pushforward measures of functions in $W^{1,p}$.
	Conditions~\cref{e:marginal-constraints-mu,e:marginal-constraints-nu} follow from \cref{e:mu-integral,e:nu-integral} when $h$ and $\ell$ depend only on $x$.
	The same two identities applied to the integral constraint~\cref{e:constraints-general-1} yield~\cref{e:mu-nu-equalities}. To see that $\supp(\mu) \subseteq \Gamma$, assume by contradiction that there exists a set
	$K \subset \Omega \times \R^m \times \R^{m\times n}$
	such that $\mu(K)>0$ and $c(x,y,z)\neq 0$ on $K$. Then, writing $\chi_{K}$ for the indicator function of $K$, we obtain the contradiction
	\begin{equation}
	    0 <
	    \langle \chi_K, \, \mu \rangle
	    = \int_{\Omega} \chi_K(x,u(x),\nabla u(x)) \dVolume = 0.
	\end{equation}
	The last equality holds because $u$ satisfies~\cref{e:constraints-general-2} and, therefore, $\chi_K(x,u(x),\nabla u(x))=0$ a.e. on $\Omega$. A similar argument proves that $\supp(\nu) \subseteq \Lambda$.
\end{proof}

A second group of constraints is derived from the divergence theorem and, loosely speaking, encodes the fact that occupation and boundary measures are constructed using $u$ as well as its gradient $\nabla u$. Define the vector space
\begin{align}\label{e:Phi-space}
	\phiSpace{p} := \bigg\{
	\phi \in C^1(\overline{\Omega} \times \R^m; \R^n):\;
	&\sum_{i=1}^n\sum_{j=1}^m \abs{\tfrac{\partial}{\partial y_j}\phi_i(x,y)} \lesssim 1+\abs{y}^{p-1},\\[-1ex]
	\nonumber
	&\sum_{i=1}^n \abs{\phi_i(x,y)} + \sum_{i=1}^n \abs{\tfrac{\partial}{\partial x_i}\phi_i(x,y)} \lesssim 1+\abs{y}^p
	\bigg\},
\end{align}
where $f \lesssim g$ means $f \leq Cg$ for some constant $C>0$. This space is clearly not empty; for example, it contains all degree-$p$ polynomials of $y$ whose coefficients depend on $x$ and are uniformly bounded on $\overline{\Omega}$.
For any $\phi \in \phiSpace{p}$, define the \textit{total divergence} $\D\phi: \Omega \times \R^m \times  \R^{m\times n} \to \R$ as
\begin{equation}\label{e:total-divergence}
\D\phi(x,y,z) := \sum_{i=1}^{n} \frac{\partial}{\partial x_i}\phi_i(x,y) + \sum_{i=1}^{n} \sum_{j=1}^{m}\frac{\partial}{\partial y_j}  \phi_i(x,y) \, z_{ji}.
\end{equation}
This definition is such that $\D\phi(x,u(x),\nabla u(x))= \nabla \cdot \phi(x,u(x))$ when the right-hand side is calculated using the chain rule. The divergence theorem then gives
\begin{equation}\label{e:div-theorem}
\int_{\Omega} \mathcal{D}\phi(x,u,\nabla u, s) \,\dVolume = \int_{\partial\Omega} \phi(x,u, s) \cdot \hat{n}(x) \,\dSurf,
\end{equation}
where $\hat{n}(x)$ is the outward unit vector normal to the boundary of $\Omega$ at $x$. Applying \cref{e:mu-integral,e:nu-integral} to this identity we find the second group of constraints on occupation and boundary measures generated by Sobolev functions.
\begin{lemma}\label{th:divergence-constraints}
    Let $\mu$ and $\nu$ be occupation and boundary measures generated by $u \in W^{1,p}(\Omega;\R^m)$, and let $\phiSpace{p}$ be the vector space of functions in~\cref{e:Phi-space}. Then,
    \begin{equation}\label{e:divergence-constraints-measures}
        \langle \D\phi, \mu \rangle - \langle \phi \cdot \hat{n}, \nu \rangle = 0 \qquad \forall \phi \in \phiSpace{p}.
    \end{equation}
\end{lemma}

\subsection{Lower bounds via occupation measures}
\label{ss:om-relaxed-problem}

We now apply the definitions above to bound the minimum of problem~\cref{e:inf-problem-general} from below. First, observe that
\begin{equation}\label{e:occ-measure-form}
\optimalValue = \inf_{(\mu,\, \nu)}  \left\{ \langle f, \mu \rangle + \langle g, \nu \rangle \right\},
\end{equation}
where the minimization is over all pairs $(\mu,\nu)$ of occupation and boundary measures generated by functions $u$ admissible for~\cref{e:inf-problem-general}.
To bound $\optimalValue$ from below, we simply extend the minimization to all pairs $(\mu,\nu)$ in the convex set defined by the conditions in \cref{th:simple-constraints,th:divergence-constraints}, which we call occupation and boundary measures. (Note our terminology departs slightly from \cite{Korda2018,Korda2022}, where such measures are called \emph{relaxed} occupation and boundary measures.)

Precisely, let $\mathcal{M}^p(\Gamma)$ and $\mathcal{M}^p(\Lambda)$ be the convex cones of Radon measures that are supported on the sets $\Gamma$ and $\Lambda$ defined in~\cref{e:Gamma-def,e:Lambda-def}, and whose moments of order $p$ or less are bounded as in \cref{e:p-moment-bounds}. Let
\begin{equation}\label{e:M-space}
    \measSpace :=  \mathcal{M}^p(\Gamma) \times \mathcal{M}^p(\Lambda).
\end{equation}
The set of \emph{occupation and boundary measures} is defined here as
\begin{equation}\label{e:A-set-definition}
\omSpace{p} := \left\{ (\mu, \nu) \in \measSpace: \;
\text{\cref{e:mu-nu-equalities,e:marginal-constraints-mu,e:marginal-constraints-nu} and \cref{e:divergence-constraints-measures}}
\right\}.
\end{equation}
It is clearly convex, as it is defined using linear and one-sided constraints, and it contains all occupation and boundary measures generated by Sobolev functions $u$ admissible in~\cref{e:inf-problem-general}.

Extending the minimization in~\cref{e:occ-measure-form} to $\omSpace{p}$ yields an infinite-dimensional linear program whose minimum bounds $\optimalValue$ from below:
\begin{equation}
\label{e:occ-measures-minimization}
\optimalValue
\geq \inf_{(\mu,\nu) \in \omSpace{p}} \; \left\{\langle f, \mu \rangle + \langle g, \nu \rangle\right\}
=: \occBound.
\end{equation}
The rest of this article analyzes the \emph{occupation measure relaxation} \cref{e:occ-measures-minimization} with an eye towards understanding whether it is or is not sharp, i.e., if $\occBound=\optimalValue$ or not.

\section{A dual scheme for proving lower bounds}\label{s:duality}

Calculating the occupation measure relaxation bound $\occBound$ amounts to solving an infinite-dimensional linear program posed over measures. In principle, this admits a dual problem posed over continuous functions. The advantage of this dual problem is that any admissible choice of functions proves a lower bound on $\occBound$. \Cref{ss:weak-duality} derives the dual problem via a minimax argument, leading to a `weak' duality result. \Cref{ss:strong-duality} improves this to a `strong' duality under suitable conditions. When strong duality holds, $\occBound$ is computed by the dual maximization.

\subsection{Weak duality}\label{ss:weak-duality}
Consider the vector space of continuous functions
\begin{equation}
    \mathbb{V} := \phiSpace{p}
	\times \R
	\times C(\overline{\Omega})
	\times C(\partial\Omega),
	\label{e:V-space}
\end{equation}
where $\phiSpace{p}$ is defined in \cref{e:Phi-space}. Consider also the convex subset
\begin{align}\label{e:B-set}
	\mathbb{B} :=
	\{(\phi,\eta,h,\ell) \in \mathbb{V}:\;
	F^{\phi, \eta, h}(x,y,z) &\geq 0 \; \text{on } \Gamma,
	\;
	G^{\phi, \eta, \ell}(x,y) &\geq 0 \; \text{on }  \Lambda
	\},
\end{align}
where
\begin{subequations}
    \begin{align}
    \label{e:F-def}
    F^{\phi, \eta, h}(x,y,z) &:= f(x,y,z) + \mathcal{D}\phi(x,y,z) - \eta \cdot a(x,y,z) - h(x),
    \\
    \label{e:G-def}
    G^{\phi, \eta, \ell}(x,y) &:= g(x,y) - \phi(x,y) \cdot \hat{n}(x) - \eta \cdot b(x,y) - \ell(x).
    \end{align}
\end{subequations}
A straightforward minimax argument yields the following weak duality result. (See \cite{Chernyavskiy2020} for a derivation of the dual problem that avoids the use of measures.)

\begin{theorem}\label{th:weak-duality}
    The optimal value $\occBound$ of the linear program~\cref{e:occ-measures-minimization} satisfies
    \begin{equation}\label{e:pdr-bound}
        \occBound \geq
        \sup_{(\phi,\eta,h,\ell)\in\mathbb{B}}
        \left\{ \int_\Omega h(x) \dVolume + \int_{\partial\Omega} \ell(x) \dSurf  \right\}
        =: \pdrBound.
    \end{equation}
\end{theorem}

\begin{proof}[Proof of \cref{th:weak-duality}]
	Define a Lagrangian function $\LagrangeFunction: \measSpace \times \mathbb{V} \to \mathbb{R}$ as
	\begin{equation}\label{e:Lagrangian-def}
	\LagrangeFunction[(\mu,\nu), (\phi,\eta,h,\ell)] :=
	\langle F^{\phi, \eta, h}, \mu \rangle + \langle G^{\phi, \eta, \ell}, \nu \rangle
	+ \int_{\Omega} h(x)\,\dVolume
	+ \int_{\partial\Omega} \ell(x)\,\dSurf.
	\end{equation}
	Recalling the definition of $\occBound$ from~\cref{e:occ-measures-minimization}, it suffices to establish that
	\begin{subequations}
	\begin{align}
	\label{e:weak-duality-1}
	\inf_{(\mu,\nu) \in \omSpace{p}}
	\,
	\left\{\langle f, \mu \rangle + \langle g, \nu \rangle\right\}
	&=\;\;
	\inf_{(\mu,\nu) \in \measSpace}
	\;
	\sup_{(\phi,\eta,h,\ell) \in \mathbb{V}}
	\;
	\LagrangeFunction[(\mu,\nu), (\phi,\eta,h,\ell)]
	\\
	\label{e:weak-duality-2}
	&\geq
	\sup_{	(\phi,\eta,h,\ell) \in \mathbb{V} }
	\;
	\inf_{(\mu,\nu) \in \measSpace}
	\;
	\LagrangeFunction[(\mu,\nu), (\phi,\eta,h,\ell)]
	\phantom{\int_{\Omega}}
	\\
	\label{e:weak-duality-3}
	&=
	\sup_{ (\phi,\eta,h,\ell) \in \mathbb{B} }
	\,
	\left\{
	\int_{\Omega} h(x)\,\dVolume +
	\int_{\partial\Omega} \ell(x)\,\dSurf
	\right\}.
	\end{align}
	\end{subequations}

	The equality in~\cref{e:weak-duality-1} holds because the supremum on the right-hand side is equal to $\langle f, \mu \rangle + \langle g, \nu \rangle$ if $(\mu,\nu) \in \omSpace{p}$, whereas it is infinite otherwise.
	The former statement is an immediate consequence of the definition of $\LagrangeFunction$ and of the constraints defining $\omSpace{p}$.
	For the latter, observe that at least one condition from~\cref{e:mu-nu-equalities}, \cref{e:marginal-constraints-mu}, \cref{e:marginal-constraints-nu} and~\cref{e:divergence-constraints-measures} fails when $(\mu,\nu)\notin \omSpace{p}$.
	Assume for definiteness that the condition being violated is~\cref{e:divergence-constraints-measures}; similar arguments apply to the other cases.
	Then, there exists $\phi_0 \in \phiSpace{p}$ such that $\langle \D\phi_0,\mu \rangle - \langle \phi_0 \cdot \hat{n},\nu \rangle = C \neq 0$, and we may take $C > 0$ by replacing $\phi_0$ with $-\phi_0$ if necessary. Since $(k\phi_0,0,0,0)$ is in $\mathbb{V}$ for all $k \in \R$,
	\begin{align*}
	\sup_{ (\phi,\eta,h,\ell) \in \mathbb{V} }
	\;
	\LagrangeFunction[(\mu,\nu), (\phi,\eta,\xi,h,\ell)]
	&\geq \LagrangeFunction[(\mu,\nu), (k\,\phi_0,0,0,0)]\\[-2ex]
	&= \langle f + k\,\D\phi_0, \mu \rangle + \langle g - k\,\phi_0 \cdot n, \nu \rangle\\
	&= \langle f, \mu \rangle + \langle g, \nu \rangle + k C.
	\end{align*}
	Letting $k \to +\infty$ shows that the supremum on the left-hand side is unbounded, as claimed. Identity~\cref{e:weak-duality-1} is therefore proven.

	The inequality in~\cref{e:weak-duality-2} is a straightforward consequence of exchanging the order of minimization and maximization, leaving only the verification of~\cref{e:weak-duality-3}. For this, recall the definition of $\LagrangeFunction$ from~\cref{e:Lagrangian-def} and of the set $\mathbb{B}$ from~\cref{e:B-set}, and observe that
	\begin{equation*}
	\inf_{(\mu,\nu) \in \measSpace} \LagrangeFunction
	=
	\begin{cases}
	\int_{\Omega} h(x)\,\dVolume +
	\int_{\partial\Omega} \ell(x)\,\dSurf & \text{if } F^{\phi, \eta, h}\geq 0 \text{ on } \Gamma \text{ and } G^{\phi, \eta, \ell} \geq 0 \text{ on } \Lambda ,\\
	-\infty &\text{otherwise}.
	\end{cases}
	\end{equation*}
	Indeed, the infimum is attained when $\mu$ and $\nu$ are the zero measures if $F^{\phi, \eta, h}$ and $G^{\phi, \eta, \ell}$ are nonnegative on $\Gamma$ and $\Lambda$, respectively. Otherwise, let $(\mu_k, \nu_k)$ be a sequence of Dirac measures with mass $k$ supported at a point of $\Gamma$ where $F^{\phi, \eta, h}<0$ and at a point of $\Lambda$ where
	$G^{\phi, \eta, \ell}<0$, and let $k \to +\infty$. \Cref{th:weak-duality} is proved.
\end{proof}

\subsection{Strong duality}\label{ss:strong-duality}
The weak duality achieved above provides lower bounds on $\occBound$ through the choice of admissible functions for \cref{e:pdr-bound}, but does not guarantee that these bounds are sharp (i.e., that $\occBound = \pdrBound$).
So far, such sharpness has been reported for particular problems where in fact $\pdrBound=\optimalValue$, including variational problems evaluating eigenvalues of Sturm--Liouville problems and optimal constants of Poincar\'e inequalities~\cite{Chernyavskiy2020}. To help separate the issues of computing $\occBound$ versus $\optimalValue$, here we prove that $\occBound = \pdrBound$ under general conditions. We turn to the question of whether $\optimalValue = \occBound$ in \cref{s:convex-integrands}.

\begin{theorem}\label{th:strong-duality-noncompact}
    The equality $\occBound = \pdrBound$ holds if there exists $\phi_0 \in \phiSpace{p}$, constants $q,r$ with $0 \leq r < q \leq p$, and a constant $\beta > 0$ such that
    \begin{subequations}
        \label{e:strong-duality-conditions-all}
        \begin{align}
        \label{e:strong-duality-conditions-coercivity-bulk}
            f(x,y,z) + \D\phi_0(x,y,z) \geq \beta \left( \abs{y}^q + \abs{z}^q - 1 \right) &\quad\text{on } \Gamma,\\
        \label{e:strong-duality-conditions-coercivity-boundary}
            g(x,y) - \phi_0(x,y) \cdot \hat{n}(x) \geq \beta \left( \abs{y}^q - 1 \right) &\quad\text{on } \Lambda,\\
        \label{e:strong-duality-conditions-growth-a}
            \abs{a(x,y,z)}  \leq \beta \left(\abs{y}^r + \abs{z}^r\right) &\quad \text{on } \Gamma,\\
        \label{e:strong-duality-conditions-growth-b}
            \abs{b(x,y)} \leq \beta \abs{y}^r  &\quad \text{on } \Lambda.
        \end{align}
    \end{subequations}
\end{theorem}

Before giving the proof, let us pause to discuss the role of $\phi_0$. So far, we have yet to invoke any sort of `coercivity' hypotheses on the integrands $f$ and $g$ from the original minimization  \cref{e:inf-problem-general}. If $f$ and $g$ are \emph{coercive}, in the sense that \cref{e:strong-duality-conditions-coercivity-bulk,e:strong-duality-conditions-coercivity-boundary} hold with $\phi_0=0$, then strong duality follows immediately from our $r$-growth assumptions on $a$ and $b$. A subcase of this that has received prior attention~\cite{Korda2018} is when the constraints from \cref{e:inf-problem-general} give compact $\Gamma$ and $\Lambda$. 
Here, we include $\phi_0$ to allow for `translations' of the integrands $f$ and $g$ that achieve coercivity while preserving the value of the functional from~\cref{e:inf-problem-general}. (That this value is preserved under the replacement $(f,g)\mapsto (f+\D\phi_0,\,g-\phi_0\cdot\hat{n})$ follows from the divergence theorem.) We imagine such translations could be useful for verifying that $\optimalValue>-\infty$ in cases where the original integrand is not evidently bounded from below.

We turn now to proving \cref{th:strong-duality-noncompact}. The proof is somewhat technical and the theorem is not used in the rest of the paper, so the reader who wishes to skip forward to the problem of proving that $\occBound=\optimalValue$ may proceed to \cref{s:convex-integrands}.

\begin{proof}[Proof of \cref{th:strong-duality-noncompact}]

Let $C^1_c(\Omega \times \R^m; \R^n)$ be the space of continuously differentiable $n$-variate functions with compact support on $\Omega \times \R^m$. Define
\begin{subequations}
\begin{align}
    \phiSpace{p}_c &:= \phiSpace{p} \cap C^1_c(\Omega \times \R^m; \R^n),\\
    \mathbb{V}_c &:=  \phiSpace{p}_c
	\times \R
	\times C(\overline{\Omega})
	\times C(\partial\Omega).
\end{align}
\end{subequations}
Further, let $\omSpace{p}_c \subset \measSpace$ be the set of measures obtained from the measure set $\omSpace{p}$ defined in~\cref{e:A-set-definition} by relaxing the `divergence theorem' condition~\cref{e:divergence-constraints-measures} to hold only for compactly supported $\phi \in \phiSpace{p}_c$. We shall prove that
\begin{subequations}
\begin{align}
    \pdrBound
    &\geq
    \sup_{ \substack{(\phi,\eta,h,\ell)\in\mathbb{B} \\ \phi - \phi_0 \in \phiSpace{p}_c} }
    \left\{ \int_\Omega h(x) \dVolume + \int_{\partial\Omega} \ell(x) \dSurf  \right\}
    \label{e:strong-duality-noncompatc-step-1}
    \\
    &= \inf_{(\mu,\nu) \in \omSpace{p}_c} \left\{\langle f + \D\phi_0, \mu \rangle + \langle g - \phi_0 \cdot \hat{n}, \nu \rangle\right\}
    \label{e:strong-duality-noncompatc-step-2}
    \\
    &= \occBound.
    \label{e:strong-duality-noncompatc-step-3}
\end{align}
\end{subequations}
This, combined with the inequality $\occBound \geq \pdrBound$ from \cref{th:weak-duality}, implies the desired identity $\occBound = \pdrBound$. Inequality~\cref{e:strong-duality-noncompatc-step-1} is immediate, while identities \cref{e:strong-duality-noncompatc-step-2,e:strong-duality-noncompatc-step-3} are proven separately below.

\paragraph{Proof of \cref{e:strong-duality-noncompatc-step-2}}
Reasoning as in the proof of \cref{th:weak-duality} shows that the maximization on the right-hand side of~\cref{e:strong-duality-noncompatc-step-1} and the minimization on the right-hand side of~\cref{e:strong-duality-noncompatc-step-2} are weakly dual problems. The equality in~\cref{e:strong-duality-noncompatc-step-2} expresses strong duality and holds if the order of minimization and maximization in the minimax problem
\begin{equation}
    \inf_{(\mu,\nu) \in \measSpace}
    \sup_{(\phi,\eta,h,\ell) \in \mathbb{V}_c}
    \!\!\LagrangeFunction[(\mu,\nu), (\phi_0 + \phi,\eta,h,\ell)]
\end{equation}
is irrelevant, where the Lagrangian function $\LagrangeFunction$ is defined in~\cref{e:Lagrangian-def}.
This fact can be established with the help of a minimax theorem by Brezis, Stampacchia and Nirenberg~\cite{Brezis1972}.
Here, we use a version of theorem given in~\cite[Theorem 5.2.2]{Nirenberg2001}, which requires checking the following conditions:
\begin{enumerate}[(BNS\arabic{*}),leftmargin=!,align=left,labelwidth=\widthof{(BNS5)},topsep=0.75ex,itemsep=0.5ex]
	\item\label{bns:vs} The set $\mathbb{V}_c$ is a convex subset of a real vector space.
	\item\label{bns:hausdorff} The set $\measSpace$ is a convex subset of a Hausdorff topological vector space.
	\item\label{bns:usc-qc} For all $(\mu,\nu) \in \measSpace$, the function $(\phi,\eta,h,\ell) \mapsto \LagrangeFunction[(\mu,\nu),(\phi,\eta,h,\ell)]$ is upper semicontinuous and quasiconcave\footnote{In this section, a real-valued function $f$ is said to be quasiconvex (resp. quasiconcave) if the pre-image of the interval $(-\infty,a)$ is a convex (resp. concave) set for every $a \in \R$. This definition arises in convex analysis and should not be confused with the completely different notion of quasiconvexity encountered in relaxation theory, which will be used in \cref{ss:failure-double-well}.} on the intersection of $\mathbb{V}_c$ with any finite-dimensional space.
	\item\label{bns:lsc-qc} For all $(\phi,\eta,h,\ell) \in \mathbb{V}_c$, the function $(\mu,\nu) \mapsto \LagrangeFunction[(\mu,\nu),(\phi,\eta,h,\ell)]$ is lower semicontinuous and quasiconvex on $\measSpace$.
	\item\label{bns:compactness} There exists $(\tilde{\phi},\tilde{\eta},\tilde{h},\tilde{\ell}) \in \mathbb{V}_c$ and a constant $\kappa$ satisfying
	$\kappa > \smash{\sup_{\mathbb{V}_c} \inf_{\measSpace} \LagrangeFunction}$
	such that the set $\{(\mu,\nu) \in \measSpace: \LagrangeFunction[(\mu,\nu), (\tilde{\phi},\tilde{\eta},\tilde{h},\tilde{\ell})] \leq \kappa \}$ is compact.
\end{enumerate}

For conditions \ref{bns:vs} and \ref{bns:hausdorff}, observe that $\mathbb{V}_c$ is a real vector space by construction, while $\measSpace$ is a convex subset of the product of the spaces $\mathcal{M}^p(\Gamma)$ and $\mathcal{M}^p(\Lambda)$ of signed Radon measures supported on $\Gamma$ and $\Lambda$ and with bounded moments of order $p$. This is a Hausdorff space when endowed with the product \weakstar\ topology. Recall that a sequence of measures $\mu_k \in \mathcal{M}^p(\Gamma)$ is said to converge \weakstar\ to a measure $\mu \in \mathcal{M}^p(\Gamma)$ if
\begin{equation}
	\langle \psi, \mu_k \rangle \to \langle \psi, \mu \rangle
\end{equation}
for all functions $\psi \in C(\Gamma)$ with compact support. In this case, we write $\mu_k \wkstarto \mu$.

For condition \ref{bns:usc-qc}, let $\mathbb{V}_d \subset \mathbb{V}$ be any $d$-dimensional space and expand its elements as linear combinations of $d$ fixed basis functions. The function $(\phi,\eta,h,\ell) \mapsto \LagrangeFunction[(\mu,\nu),(\phi,\eta,h,\ell)]$ is affine (hence, quasiconvex) in the $d$ expansion coefficients. It is also continuous (hence, upper semicontinuous)  in the usual topology on $\R^d$.

For condition \ref{bns:lsc-qc}, note that the function $(\mu,\nu) \mapsto \LagrangeFunction[(\mu,\nu),(\phi_0+\phi,\eta,h,\ell)]$ is affine for every fixed $(\phi,\eta,h,\ell) \in \mathbb{V}_c$, so in particular it is quasiconvex. To establish its \weakstar\ lower semicontinuity, instead, we need to show that
\begin{subequations}
\begin{align}
    \label{e:F-wlsc}
    \liminf_{\mu_k \wkstarto \mu } \langle F^{\phi_0+\phi, \eta, h}, \mu_k \rangle &\geq \langle F^{\phi_0+\phi, \eta, h}, \mu \rangle,
    \\
    \label{e:G-wlsc}
    \liminf_{\nu_k \wkstarto \nu } \langle G^{\phi_0+\phi, \eta, \ell}, \nu_k \rangle &\geq \langle G^{\phi_0+\phi, \eta, \ell}, \nu \rangle.
\end{align}
\end{subequations}
Since $\phi$ has compact support and $h$ is continuous on $\overline{\Omega}$, we can use the assumed coercivity of $f+\D\phi_0$ in~\cref{e:strong-duality-conditions-coercivity-bulk} and the bounded growth of $a(x,y,z)$ from~\cref{e:strong-duality-conditions-growth-a} to conclude that $F^{\phi_0+\phi, \eta, h}$ is coercive on $\Gamma$, i.e.,
\begin{equation*}
    F^{\phi_0+\phi, \eta, h}(x,y,z)
    = [f + \D\phi_0 + \D\phi - \eta a - h](x,y,z)
    \gtrsim \abs{y}^q + \abs{z}^q - 1 \quad \text{on }\Gamma.
\end{equation*}
In particular, the negative part $F_{-}^{\phi_0+\phi, \eta, h}$ of $F^{\phi_0+\phi, \eta, h}$ is compactly supported on the (possibly noncompact) support $\Gamma \subseteq \overline{\Omega} \times \R^m \times \R^{m\times n}$ of the measures $\mu_k$ and $\mu$. Then, writing $\smash{F_{+}^{\phi_0+\phi, \eta, h}}$ for the positive part of $\smash{F^{\phi_0+\phi, \eta, h}}$ and using the definition of \weakstar\ convergence of measures,
\begin{align*}
    \liminf_{\mu_k \wkstarto \mu } \langle F^{\phi_0+\phi, \eta, h}, \mu_k \rangle
    &= \liminf_{\mu_k \wkstarto \mu } \left[ \langle F_+^{\phi_0+\phi, \eta, h}, \mu_k \rangle - \langle F_-^{\phi_0+\phi, \eta, h}, \mu_k \rangle\right]
    \\ \nonumber
    &= \liminf_{\mu_k \wkstarto \mu } \langle F_+^{\phi_0+\phi, \eta, h}, \mu_k \rangle -  \langle F_-^{\phi_0+\phi, \eta, h}, \mu \rangle.
\end{align*}
To obtain~\cref{e:F-wlsc}, let $\{\chi_m\}_{m\geq 0}$ be a partition of unity of $\Gamma$ and estimate
\begin{align*}
    \liminf_{\mu_k \wkstarto \mu } \langle F_+^{\phi_0+\phi, \eta, h}, \mu_k \rangle
    &= \liminf_{\mu_k \wkstarto \mu } \left\langle \sum_{m\geq 0} \chi_m F_+^{\phi_0+\phi, \eta, h} , \mu_k \right\rangle
    \\ \nonumber
    &= \liminf_{\mu_k \wkstarto \mu } \sum_{m\geq 0} \langle\chi_m F_+^{\phi_0+\phi, \eta, h}, \mu_k \rangle
    \\ \nonumber
    \text{(Fatou's lemma)}
    &\geq \sum_{m\geq 0} \liminf_{\mu_k \wkstarto \mu } \langle \chi_m F_+^{\phi_0+\phi, \eta, h}, \mu_k \rangle
    \\ \nonumber
    \text{(\weakstar\ convergence)}
    &= \sum_{m\geq 0} \langle \chi_m F_+^{\phi_0+\phi, \eta, h}, \mu \rangle
    \\ \nonumber
    &= \langle F_+^{\phi_0+\phi, \eta, h}, \mu \rangle.
\end{align*}
Inequality~\cref{e:G-wlsc} is proved using nearly identical steps based on assumptions \cref{e:strong-duality-conditions-coercivity-boundary} and \cref{e:strong-duality-conditions-growth-b}, which we omit for brevity. Condition~\ref{bns:lsc-qc} is therefore established.

There remains to verify condition~\ref{bns:compactness}. First, we fix a suitable $\kappa$. Inequalities \cref{e:occ-measures-minimization} and \cref{e:strong-duality-noncompatc-step-1} and steps similar to those in the proof of \cref{th:weak-duality} imply that
\begin{equation*}
\optimalValue
\geq
\sup_{(\phi,\eta,h,\ell) \in \mathbb{V}_c} \inf_{(\mu,\nu) \in \measSpace} \LagrangeFunction[(\mu,\nu), (\phi+\phi_0,\eta,h,\ell)],
\end{equation*}
where $\optimalValue$ is the global minimum of the variational problem~\cref{e:inf-problem-general}. We may assume that $\optimalValue > -\infty$, otherwise \cref{th:strong-duality-noncompact} holds trivially with $\occBound = \pdrBound = -\infty$.
Then, $\kappa=1+\optimalValue$ satisfies the strict inequality required by condition~\ref{bns:compactness}.
Next, we set $(\tilde{\phi},\tilde{\eta},\tilde{h},\tilde{\ell})=(0,0,-1-\beta,-1-\beta) \in \mathbb{V}_c$, where $\beta$ is the constant appearing in~\cref{e:strong-duality-conditions-all}, and check the \weakstar\ compactness of the set
\begin{equation*}
\mathbb{K} := \{(\mu,\nu) \in \mathbb{M}: \LagrangeFunction[(\mu,\nu), (\phi_0,0,-1-\beta,-1-\beta)] \leq 1+\optimalValue \}.
\end{equation*}
This follows from Prokhorov's theorem (see, e.g.,~\cite[Theorem~8.6.2]{Bogachev2007}) if we can show that $\mathbb{K}$ is sequentially \weakstar\ closed, bounded in the total variation norm, and uniformly tight---meaning that, for every $\varepsilon>0$, there exist compact sets $\Lambda_\varepsilon \subseteq \Lambda$ and $\Gamma_\varepsilon \subseteq \Gamma$ such that $\mu(\Gamma \setminus \Gamma_\varepsilon) + \nu(\Lambda \setminus \Lambda_\varepsilon) \leq \varepsilon$ whenever
$(\mu,\nu) \in \mathbb{K}$.
Sequential closedness follows from the \weakstar\ lower semicontinuity of the function $(\mu,\nu) \mapsto \LagrangeFunction[(\mu,\nu), (\phi_0,0,-1-\beta,-1-\beta)]$, which was proven above. For boundedness and uniform tightness, instead, observe that the estimates in~\cref{e:strong-duality-conditions-all} imply
\begin{align}\label{e:noncompact-proof-estimates-1}
    1+\optimalValue
    &> \LagrangeFunction[(\mu,\nu), (\phi_0,0,-1-\beta,-1-\beta)]
    \\ \nonumber
    &= \langle f + \D\phi_0 + 1+\beta, \mu \rangle + \langle g - \phi_0 \cdot \hat{n} + 1+\beta, \nu \rangle
    \\ \nonumber
    &\geq \langle \beta\abs{y}^q + \beta\abs{z}^q + 1, \mu \rangle + \langle \beta\abs{y}^q + 1, \nu \rangle
\end{align}
for all $(\mu,\nu) \in \mathbb{K}$. The last expression is bounded below by $\langle 1, \mu \rangle + \langle 1, \nu \rangle  = \|(\mu,\nu)\|$, proving that $\mathbb{K}$ is bounded in the total variation norm.
Moreover, given any $\varepsilon>0$, one can consider the compact sets
$\Gamma_\varepsilon = \Gamma \cap \{(x,y,z): \abs{y}^q + \abs{z}^q \leq \varepsilon^{-1}\beta^{-1}(1+\optimalValue)\}$  and
$\Lambda _\varepsilon = \Lambda \cap \{(x,y): \abs{y}^q \leq \varepsilon^{-1}\beta^{-1}(1+\optimalValue)\}$
and drop the positive terms $\langle 1,\mu\rangle+\langle1,\nu\rangle$ from the right-hand side of~\cref{e:noncompact-proof-estimates-1} to obtain
\begin{align*}
  \beta^{-1}(1+\optimalValue)
  &> \langle \abs{y}^q + \abs{z}^q, \mu \rangle + \langle \abs{y}^q , \nu \rangle
  \\ \nonumber
  &\geq \int_{\Gamma \setminus \Gamma_\varepsilon} \abs{y}^q + \abs{z}^q \,{\rm d}\mu + \int_{\Lambda \setminus \Lambda_\varepsilon} \abs{y}^q \,{\rm d}\nu
  \\ \nonumber
  &\geq \varepsilon^{-1}\beta^{-1}(1+\optimalValue) \left[ \mu(\Gamma \setminus \Gamma_\varepsilon)  + \nu(\Lambda \setminus \Lambda_\varepsilon) \right].
\end{align*}
Rearranging this inequality shows that the set of measures $\mathbb{K}$ is uniformly tight.
Condition~\ref{bns:compactness} is therefore verified, concluding the proof of identity~\cref{e:strong-duality-noncompatc-step-2}.

\paragraph{Proof of \cref{e:strong-duality-noncompatc-step-3}}
We use a truncation argument to prove that if the `divergence theorem' identity $\langle \D\phi,\mu \rangle = \langle \phi \cdot \hat{n}, \nu\rangle$ holds for all compactly supported $\phi\in\phiSpace{p}_c$, then it holds for general $\phi \in \phiSpace{p}$.
This implies $\omSpace{p}_c \equiv \omSpace{p}$ and establishes~\cref{e:strong-duality-noncompatc-step-3} because $\langle f + \D\phi_0, \mu \rangle + \langle g - \phi_0 \cdot \hat{n}, \nu \rangle = \langle f, \mu \rangle + \langle g , \nu \rangle$  for all $(\mu,\nu) \in \omSpace{p}$.

It suffices to show that, for every $\phi \in \phiSpace{p}$, every pair $(\mu,\nu) \in \measSpace$, and every $\varepsilon>0$, there exists $r>0$ and $\phi_r \in \phiSpace{p}_c$ such that
\begin{subequations}
\begin{align}\label{e:truncation-step-0a}
    \abs{\langle \D\phi - \D\phi_r, \mu \rangle } &\leq \varepsilon, \\
    \abs{\langle \phi\cdot\hat{n} - \phi_r\cdot\hat{n}, \nu \rangle } &\leq \varepsilon.
    \label{e:truncation-step-0b}
\end{align}
\end{subequations}

Let $\Psi_r: \R_+ \to [0,1]$ be a smooth cut-off function defined via
\begin{equation*}
    \Psi_r(\xi) = \begin{cases}
    1 & \text{if } \xi \leq r^2\\
    \psi_{r}(\xi) & \text{if } r^2 \leq \xi \leq 4r^2\\
    0 & \text{if } \xi \geq 4r^2,
    \end{cases}
\end{equation*}
where $\psi_r: \R_+ \to [0,1]$ is a monotonically decreasing function satisfying
\begin{equation}\label{e:cut-off-rate}
    \abs{\psi_r'(\xi)} \leq 4r^{-2}.
\end{equation}
Set $\phi_r(x,y) = \Psi_r(\abs{y}^2) \phi(x,y)$ and, for every constant $\gamma>0$, define the set $\Gamma_\gamma := \{(x,y,z) \in \Gamma: \abs{y} \leq \gamma \}$. Then,
\begin{equation}\label{e:truncation-step-1}
    \abs{\langle \D\phi - \D\phi_r, \mu \rangle}
    \leq \langle \abs{\D\phi - \D\phi_r}, \mu \rangle
    \leq \mathcal{I}_1 + \mathcal{I}_2,
\end{equation}
where (summing over repeated indices to lighten the notation)
\begin{align*}
\mathcal{I}_1 &:= \int_{\Gamma \setminus \Gamma_r} \left[ 1 - \Psi_r(\abs{y}^2) \right]
    \left( \abs{ \frac{\partial \phi_i}{\partial x_i} } + \abs{ \frac{\partial \phi_i}{\partial y_j} } \abs{ z_{ji} } \right) {\rm d}\mu,
    \\
    \mathcal{I}_2&:= 2\int_{\Gamma_{2r} \setminus \Gamma_r} \abs{\Psi_r'(\abs{y}^2)} \abs{\phi_i} \abs{y_j} \abs{ z_{ji} } {\rm d}\mu.
\end{align*}
Using the inequality $\Psi_r(\abs{y}^2) \leq 1$, inequality~\cref{e:cut-off-rate}, the growth conditions on $\phi$ from~\cref{e:Phi-space}, the fact that $\abs{y} \leq 2r$ on $\Gamma_{2r} \setminus \Gamma_r$, and the Cauchy--Schwarz inequality, we can estimate
\begin{subequations}
\begin{align}
    \label{e:estimate-I1}
    \mathcal{I}_1
    &\leq \int_{\Gamma \setminus \Gamma_r}
    \abs{ \frac{\partial \phi_i}{\partial x_i} } + \abs{ \frac{\partial \phi_i}{\partial y_j} } \abs{ z_{ji} } \, {\rm d}\mu
    \lesssim \int_{\Gamma \setminus \Gamma_r}
    1 + \abs{y}^p + \abs{z} + \abs{y}^{p-1} \abs{z} \, {\rm d}\mu,
    \intertext{and}
    \mathcal{I}_2
    &\lesssim \int_{\Gamma_{2r} \setminus \Gamma_r} r^{p-1} \abs{z} {\rm d}\mu
    \lesssim r^{p-1} \bigg( \int_{\Gamma_{2r} \setminus \Gamma_r} \abs{z}^p {\rm d}\mu \bigg)^{\frac1p} \abs{\mu(\Gamma_{2r} \setminus \Gamma_r)}^{\frac{p-1}{p}}.
    \label{e:estimate-I2}
\end{align}
\end{subequations}
Moreover, since the (fixed) measure $\mu$ has bounded moments of degree $p$,
\begin{equation*}
    \mu(\Gamma_{2r} \setminus \Gamma_r)
    \leq \mu(\Gamma \setminus \Gamma_r)
    \leq \frac{1}{r^p} \int_{\Gamma \setminus \Gamma_r} r^{p} \,{\rm d}\mu
    \leq \frac{1}{r^p} \int_{\Gamma} \abs{y}^{p} \,{\rm d}\mu
    \sim \frac{1}{r^p}.
\end{equation*}
Thus, we can replace~\cref{e:estimate-I2} with
\begin{equation}\label{e:estimate-I2-better}
    \mathcal{I}_2
    \lesssim \bigg( \int_{\Gamma_{2r} \setminus \Gamma_r} \abs{z}^p {\rm d}\mu \bigg)^{\frac1p}
    \lesssim \bigg( \int_{\Gamma \setminus \Gamma_r} \abs{z}^p {\rm d}\mu \bigg)^{\frac1p}.
\end{equation}
Now, the right-hand sides of~\cref{e:estimate-I1,e:estimate-I2-better} tend to zero as $r$ tends to infinity because $\Gamma \setminus \Gamma_r$ increases to $\Gamma$ and $\mu$ has bounded moments of degree $p$ or less. In particular, there exists $r$ large enough that $\mathcal{I}_1 + \mathcal{I}_2 \leq \varepsilon$. Combining this with~\cref{e:truncation-step-1} yields~\cref{e:truncation-step-0a}, as desired.
Inequality~\cref{e:truncation-step-0b} follows from similar estimates using the fact that $\abs{\hat{n}}=1$, which are omitted for brevity. Identity~\cref{e:strong-duality-noncompatc-step-3}, and therefore \cref{th:strong-duality-noncompact}, are proved.
\end{proof}

\section{Sharpness of occupation measure bounds}\label{s:convex-integrands}

We now ask whether $\occBound$ and $\pdrBound$ are sharp lower bounds on the optimal value $\optimalValue$ of the original minimization problem~\cref{e:inf-problem-general}, i.e., whether $\optimalValue = \occBound = \pdrBound$. These identities are proved in~\cite{Chernyavskiy2020} for particular classes of quadratic problems. Here, we establish a different sharpness result covering problems in the form
\begin{equation}\label{e:convex-integrands}
    \optimalValue = \inf_{ u \in W_0^{1,p}(\Omega; \R^m) } \int_{\Omega} f_0(x,\nabla u) + f_1(x,u) \, \dVolume
\end{equation}
where $\Omega \subset \mathbb{R}^n$ is a bounded Lipschitz domain, $u$ is subject only to the homogeneous Dirichlet boundary conditions $u=0$, and the functions $f_0$, $f_1$ are continuous in both arguments and convex in their second arguments. 

\begin{theorem}\label{th:sharpness}
    Consider the variational problem \cref{e:convex-integrands} where  $\Omega$ is a bounded Lipschitz domain, $p>1$, and the functions $f_0$ and $f_1$ are continuous. Suppose that:
    \begin{enumerate}[(H\arabic{*}),leftmargin=\parindent,align=left,labelwidth=!,topsep=0.75ex,itemsep=0.5ex]
        \item\label{ass:sharpness-convexity} The functions $z \mapsto f_0(x,z)$ and $y \mapsto f_1(x,y)$ are convex for every $x \in \Omega$.
        \item\label{ass:sharpness-growth} There exist constants $c_1, c_2, c_3 > 0$ such that
        \begin{align*}
            c_1\abs{z}^p - c_2 \leq f_0(x,z) &\leq c_3(\abs{z}^p + 1) \quad \forall x \in \Omega,\\
            c_1\abs{y}^p - c_2 \leq f_1(x,y) &\leq c_3(\abs{y}^p + 1) \quad \forall x \in \Omega.
        \end{align*}
    \end{enumerate}
    Then, $\optimalValue = \occBound = \pdrBound$. Furthermore, the maximization in \cref{e:pdr-bound} can be carried out with $\eta=0$, $\ell=0$ and $\phi \in \phiSpace{p}$ of the form $\phi(x,y)=\sigma(x)y$ with $\sigma \in \smash{C^1(\overline{\Omega};\R^{m \times n})}$ without changing its value.
\end{theorem}

\begin{remark}
	The choice $\eta=0$ in \cref{e:pdr-bound} is natural because \cref{e:convex-integrands} has no integral constraints. The choice $\ell=0$, instead, is optimal because the functional being minimized involves no boundary integrals and, as shown in the proof, there is an optimizing sequence of $\phi$ such that the trace of $\phi(x,u(x))$ vanishes when $u \in  W_0^{1,p}(\Omega; \R^m)$.
\end{remark}

\begin{remark}\label{remark:sharpness-BCs} We use the homogeneous boundary condition $u=0$ for simplicity; one can handle non-zero boundary conditions $u=u_0$ by writing the minimization in terms of $v=u-u_0$, for suitably regular $u_0$. We leave this to the reader.
\end{remark}

\begin{remark}\label{remark:linear-case}
    The coercivity and growth conditions on $f_1$ in \ref{ass:sharpness-growth} can be replaced with the assumption that $f_1(x,y) = F(x) \cdot y$ for some smooth function $F$ with straightforward changes to the proof.
\end{remark}

\begin{remark}
   A more general version of \cref{th:sharpness} appeared in~\cite{hkkz2023arxiv} after this paper was submitted. It allows for integrands $f(x,u,\nabla u)$ that are jointly convex in $u$ and $\nabla u$, as well as for certain convex constraints.
\end{remark}

The proof of \cref{th:sharpness} is given in \cref{ss:sharpness-proof} and uses three technical lemmas established in \cref{ss:sharpness-lemmas}. There, we make extensive use of the Legendre transforms
\begin{gather*}
        f_0^*(x,z^*) = \sup_{z \in \R^{m\times n}} \{ z \cdot z^* - f_0(x,z) \},\\
        f_1^*(x,y^*) = \sup_{y \in \R^m} \{ y \cdot y^* - f_1(x,y) \}.
\end{gather*}
First, however, we show that the convexity assumption on $f_0$ can be removed for one-dimensional or scalar variational problems simply by replacing $f_0$ with its convexification (double Legendre transform) $f_0^{**}$. Thus, when $n=1$ or $m=1$ the occupation measure and pointwise dual relaxations are sharp even when the integrand in \cref{e:convex-integrands} is not convex in the gradient term.

\begin{corollary}\label{th:sharpness-corollary}
	Consider the variational problem \cref{e:convex-integrands} where $\Omega$ is a bounded Lipschitz domain, $p>1$, and $n=1$ or $m=1$. If the growth conditions \ref{ass:sharpness-growth} hold and the function $y \mapsto f_1(x,y)$ is convex for every $x \in \Omega$, then $\optimalValue = \occBound = \pdrBound$.
	Moreover, the maximization in \cref{e:pdr-bound} can be restricted as explained in \cref{th:sharpness}.
\end{corollary}

\begin{remark}
	The result extends to functions $f_1$ that do not satisfy condition~\ref{ass:sharpness-growth}, but are linear in $y$ as explained in \cref{remark:linear-case}. This is the case for the one-dimensional example in \S5.3 of \cite{Chernyavskiy2020}, where numerical lower bounds on $\pdrBound$ computed using semidefinite programming agree with $\optimalValue$ to high accuracy.
\end{remark}
\begin{remark}
	The convexification argument used to prove \cref{th:sharpness-corollary} cannot be applied to remove the convexity assumption on $f_1$ without changing the minimum in general. A simple counterexample is the one-dimensional scalar problem
	\begin{equation}
		\optimalValue := \inf_{u(\pm1)=0} \int_{-1}^1 \abs{\frac{{\rm d}u}{{\rm d}x}}^2 + \abs{u-1} \abs{u+1} \, \dVolume,
	\end{equation}
	where $f_0(x,z)=z^2$ is convex but $f_1(x,y) = \abs{y-1} \abs{y+1}$ is not. It is clear that $\optimalValue$ is strictly positive.
	On the other hand, since
	\begin{equation}
		f_1^{**}(x,y) = \begin{cases}
			\abs{y-1} \abs{y+1} &\text{if } \abs{y}\geq 1\\
			0 &\text{otherwise}
		\end{cases}
	\end{equation}
	the minimum of the convexified functional $\int_{-1}^1 \abs{\frac{{\rm d}u}{{\rm d}x}}^2 + f_1^{**}(x,u) \, \dVolume$ is zero.
\end{remark}

\begin{proof}[Proof of \cref{th:sharpness-corollary}]
	It is known (see, e.g.,~\cite[Theorem~9.8]{Dacorogna2008}) that when $n=1$ or $m=1$ we can replace $f_0$ with its convexification $f_0^{**}$ without changing $\optimalValue$, i.e.
	\begin{equation*}
		\mathcal{F}^* = \inf_{ u \in W_0^{1,p}(\Omega; \R^m) } \int_{\Omega} f_0^{**}(x,\nabla u) + f_1(x,u) \, \dVolume.
	\end{equation*}
	To prove the corollary we apply \cref{th:sharpness} and the standard identity $f_0^{***} = f_0^*$. For this, we must check that $f_0^{**}$ satisfies the same growth and coercivity conditions as $f_0$, possibly with different constants $c_1$, $c_2$ and $c_3$. This follows from the definition of the Legendre transform and the growth and coercivity conditions on $f_0$. Indeed, by the coercivity of $f_0$,
	\begin{multline}
	f_0^*(x,z^*)
	:= \sup_{z \in \mathbb{R}^n} \{z^* \cdot z - f_0(x,z)\}
	\leq \sup_{z \in \mathbb{R}^n} \{z^* \cdot z - c_1\abs{z}^p\} + c_2
	\\
	= \frac1q \left( \frac{1}{ p c_1}\right)^{\frac{1}{p-1}} \abs{z^*}^q + c_2
	\leq c_4 \left( \abs{z^*}^q + 1\right),
	\end{multline}
	where $q=p/(p-1)$ and $c_4$ is the largest of the two constants appearing in the penultimate line. Similarly, the growth condition on $f_0$ gives
	\begin{multline}
	f_0^*(x,z^*)
	:= \sup_{z \in \mathbb{R}^n} \{z^* \cdot z - f_0(x,z)\}
	\\\geq \sup_{z \in \mathbb{R}^n} \{z^* \cdot z - c_3\abs{z}^p\} - c_3
	= \frac1q \left( \frac{1}{ p c_3}\right)^{\frac{1}{p-1}} \abs{z^*}^q + c_3.
	\end{multline}
	Thus, $f_0^*$ satisfies coercivity and growth conditions similar to those of $f_0$, but with different constants and the conjugate exponent $q$. By the same arguments one concludes that the double Legendre transform $f_0^{**} = (f_0^*)^*$ satisfies growth and coercivity conditions with exponent $p$, as desired.
\end{proof}

The rest of this section proves \cref{th:sharpness}.
The key is to recognize that pointwise dual relaxations solve a dual variational problem that can be obtained directly using the Legendre transform (cf. \cref{lemma:sharpness}).

\subsection{Three technical lemmas}\label{ss:sharpness-lemmas}
First, we prove that certain integral functionals on  $L^q(\Omega;\R^{m\times n})$ and $L^q(\Omega;\R^{m})$ defined using $f_0^*$ and $f_1^*$ are continuous. Here, $q$ is the H\"older conjugate of the exponent $p$ from \cref{th:sharpness}, meaning that $1/p+1/q=1$.

\begin{lemma}\label{lemma:continuity}
    Under the growth and coercivity conditions~\ref{ass:sharpness-growth}, the functions $\sigma^* \mapsto \int_\Omega f_0^*(x,\sigma^*(x)) \dVolume$ and $\rho^* \mapsto \int_\Omega f_1^*(x,\rho^*(x)) \dVolume$ are finite and strongly continuous from $L^q(\Omega;\R^{m\times n})$ and $L^q(\Omega;\R^{m})$ into $\R$.
\end{lemma}

\begin{proof}
We prove the results only for
$\sigma^* \mapsto \int_\Omega f_0^*(x,\sigma^*(x)) \dVolume$
since the arguments for $\rho^* \mapsto \int_\Omega f_1^*(x,\rho^*(x)) \dVolume$ are identical. To see that $\int_\Omega f_0^*(x,\sigma^*(x)) \dVolume < +\infty$ for all $\sigma \in L^q(\Omega;\R^{m \times n})$, observe that straightforward estimates using assumption~\ref{ass:sharpness-growth} and Young's inequality yield
$\abs{f_0^*(x,z^*)}\lesssim 1 + \abs{z^*}^q$ for all $x \in \Omega$. To establish strong continuity, instead, it suffices to check that $z^* \mapsto f_0^*(x,z^*)$ is continuous.
This can be done with the help of~\cite[Theorem~7.A]{Rockafellar1966}, which guarantees continuity if, for every $z^*$ and any $\alpha<f^*(x,z^*)$, the set $\{ z \in \R^{m \times n}: f_0(x,z) +\alpha \leq z \cdot z^* \}$  is bounded.  This is an easy consequence of the lower bound on $f_0$ in assumption~\ref{ass:sharpness-growth}, combined with the H\"older inequality $z\cdot z^* \leq |z|^p/p + |z^*|^q/q$ with  $p>1$ and $1/p+1/q=1$.
\end{proof}

Next, we use classical duality theory~\cite{EkelandTemam1999} to express $\optimalValue$ as the optimal value of a constrained maximization problem involving the Legendre transforms $f_0^*$ and $f_1^*$.
\begin{lemma}\label{lemma:sharpness}
    Let $q = \frac{p}{p-1}$. Under the assumptions of \cref{th:sharpness},
    \begin{equation}\label{e:conjugate-dual-problem}
        \optimalValue = \sup_{ \substack{ \sigma \in L^{q}(\Omega; \R^{m\times n}) \\ \rho \in L^{q}(\Omega; \R^{m}) \\ \nabla \cdot \sigma = \rho} }
    \int_\Omega -f_0^*(x, \sigma) - f_1^*(x,\rho) \, \dVolume.
    \end{equation}
\end{lemma}

\begin{proof}
Define the convex functions
\begin{align*}
    \Phi(\sigma,\rho) &:= \int_\Omega f_0^*(x,\sigma) + f_1^*(x,\rho) \dVolume,
    &\text{and} &&
    \Psi(\sigma,\rho) &:= \begin{cases}
     0 &\text{if } \nabla \cdot \sigma = \rho,\\
     +\infty &\text{otherwise},
    \end{cases}
\end{align*}
which are convex from $L^q(\Omega;\R^{m \times n})  \times L^q(\Omega;\R^{m})$ into $(-\infty,+\infty]$. Let $D(\Phi)$ and $D(\Psi)$ be the sets in $L^q(\Omega;\R^{m \times n})  \times L^q(\Omega;\R^{m})$ where they take on finite values.
The right-hand side of~\cref{e:conjugate-dual-problem} is equivalent to maximizing $-\Phi(\sigma) - \Psi(\sigma)$, and the identity
\begin{equation}\label{e:fenchel-rockafellar}
    \sup_{ \substack{\sigma \in L^q(\Omega;\R^{m \times n}) \\ \rho \in L^q(\Omega;\R^{m})} } \{-\Phi(\sigma,\rho) - \Psi(\sigma,\rho)\}
    = \inf_{ \substack{\sigma^* \in L^p(\Omega;\R^{m \times n}) \\ \rho^* \in L^p(\Omega;\R^{m})} } \{\Phi^*(-\sigma^*,-\rho^*) + \Psi^*(\sigma^*,\rho^*)\}
\end{equation}
holds if there exists $(\sigma_0,\rho_0) \in D(\Phi) \cap D(\Psi)$ at which $\Phi$ is continuous in the strong topology of $L^q(\Omega;\R^{m \times n}) \times L^q(\Omega;\R^{m})$ \cite[Theorem~1.12]{Brezis2011}.
This is true because $\Phi$ is continuous on the entire space $L^q(\Omega;\R^{m \times n}) \times L^q(\Omega;\R^{m})$ by \cref{lemma:continuity}, while $D(\Psi)$ is nonempty because the equation $\nabla \cdot \sigma= \rho$ admits at least one solution
$\sigma \in L^q(\Omega;\R^{m \times n})$ for every $\rho \in L^q(\Omega;\R^{m})$~\cite{Russ2013}.
The lemma is therefore proven if we can show that the right-hand side of~\cref{e:fenchel-rockafellar} is equal to $\optimalValue$, i.e.
\begin{equation}\label{e:composite-conjugation}
    \inf_{ \substack{\sigma^* \in L^p(\Omega;\R^{m \times n}) \\ \rho^* \in L^p(\Omega;\R^{m})} } \{\Phi^*(-\sigma^*,-\rho^*) + \Psi^*(\sigma^*,\rho^*)\}= \inf_{u \in W_0^{1,p}}  \int_\Omega f_0(x,\nabla u) + f_1(x,u) \dVolume.
\end{equation}

To establish this identity, observe that
\begin{align}\label{e:Phi-conjugate}
    \Phi^*(-\sigma^*,-\rho^*)
    &=
    \sup_{ \substack{\sigma \in L^q(\Omega;\R^{m \times n}) \\ \rho \in L^q(\Omega;\R^{m})} }
    \int_{\Omega} -\sigma^* \cdot \sigma - \rho^* \cdot \rho - f_{0}^*(x, \sigma) - f_{1}^*(x, \rho) \dVolume
    \\ \nonumber
    &= \int_\Omega f_0^{**}(x,-\sigma^*) + f_1^{**}(x,-\rho^*) \dVolume
    \\ \nonumber
    &= \int_\Omega f_0(x,-\sigma^*) + f_1(x,-\rho^*) \dVolume.
\end{align}
The second equality follows from~\cite[Proposition~IX.2.1]{EkelandTemam1999} and the last one is a consequence of assumption~\ref{ass:sharpness-convexity}, since the double Legendre transform of a convex function is the function itself.
To calculate the Legendre transform of $\Psi$, instead, observe via the Helmholtz decomposition that any $\sigma$ satisfying $\nabla \cdot \sigma = \rho$ can be written as $\sigma = \nabla v^\rho + \xi$, where $\xi$ is divergence-free and the components of $v^\rho=(v^\rho_1,\ldots,v^\rho_m) \in \smash{W^{1,p}_0}$ satisfy the Poisson equation $\Delta v^\rho_i = \rho_i$. Thus,
\begin{align*}
    \Psi^*(\sigma^*,\rho^*)
    &= \sup_{ \substack{\sigma \in L^{q}(\Omega;\R^{m \times n}) \\ \rho \in L^{q}(\Omega;\R^{m}) \\ \nabla \cdot \sigma = \rho } }
    \int_{\Omega} \sigma^* \cdot \sigma + \rho^* \cdot \rho \dVolume
    \\
    &= \sup_{ \rho \in L^{q}(\Omega;\R^{m}) }
    \sup_{ \substack{\xi \in L^{q}(\Omega;\R^{m \times n}) \\ \nabla \cdot \xi = 0 } }
    \int_{\Omega} \sigma^* \cdot \nabla v^\rho + \sigma^* \cdot \xi + \rho^* \cdot \rho \dVolume.
\end{align*}
The inner supremum is unbounded unless $\int_\Omega \sigma^* \cdot \xi \dVolume=0$, which requires $\sigma^* = -\nabla u$ for some $u \in \smash{W^{1,p}_0}$. In this case, integration by parts and the identity $\Delta v^\rho_i = \rho_i$ yield
\begin{multline*}
    \Psi^*(-\nabla u,\rho^*)
    =\sup_{ \rho \in L^{q}(\Omega;\R^{m}) }
    \int_{\Omega}  u \cdot \Delta v^\rho + \rho^* \cdot \rho \; \dVolume
    \\
    =\sup_{ \rho \in L^{q}(\Omega;\R^{m}) }
    \int_{\Omega} (u + \rho^*) \cdot \rho \; \dVolume
    = \begin{cases}
    0 &\text{if } \rho^* = -u,\\
    +\infty &\text{otherwise}.
    \end{cases}
\end{multline*}
Thus,
\begin{equation}\label{e:Psi-conjugate}
   \Psi^*(\sigma^*,\rho^*)
   = \begin{cases}
   0 &\text{if } \exists u\in W_0^{1,p}:\;\sigma^*=-\nabla u, \, \rho^* = -u,\\
   +\infty &\text{otherwise}.
   \end{cases}
\end{equation}
Combining~\cref{e:Phi-conjugate} and~\cref{e:Psi-conjugate} yields~\cref{e:composite-conjugation}, as required.
\end{proof}

Finally, we show that feasible functions for the maximization problem~\cref{e:conjugate-dual-problem} can be approximated by feasible functions that are smooth on $\overline{\Omega}$. We believe the result is standard, but give a proof for completeness.

\begin{lemma}\label{lemma:smoothness}
    Suppose $\sigma \in L^{q}(\Omega;\R^{m \times n})$ and $\rho \in L^{q}(\Omega;\R^{m})$ satisfy
	$\nabla \cdot \sigma = \rho$.
	There exists sequences
	$\{\sigma_k\}_{k \in \mathbb{N}} \subset C^{\infty}(\overline{\Omega};\R^{m \times n})$
	and
	$\{\rho_k\}_{k \in \mathbb{N}} \subset C^{\infty}(\overline{\Omega};\R^{m})$
	with
	$\nabla \cdot \sigma_k = \rho_k$
	such that
	$\| \sigma_k - \sigma \|_q \to 0$
	and
	$\|\rho_k - \rho\|_q \to 0$ as $k \to \infty$.
\end{lemma}

\begin{proof}
	Let $B$ be a ball with $\Omega \subset B$ and ${\rm dist}(\Omega,\partial B) \geq 1$, so $B\setminus \Omega$ is a bounded Lipschitz domain with boundary $\partial \Omega \cup \partial B$.
	Let $p=q/(q-1)$ and let $u$ solve
	\begin{equation*}
		\min_{ \substack{ u \in W^{1,p}(B\setminus\Omega; \mathbb{R}^m) \\ u = 0 \text{ on } \partial B} }
		\int_{B \setminus \Omega} \frac{\abs{\nabla u}^p}{p}\dVolume
		- \int_{\partial\Omega} (\sigma \cdot \hat{n}) \cdot u \dSurf.
	\end{equation*}
	This is a convex problem, so $u$ satisfies the Euler--Lagrange equation
	\begin{equation*}
		\begin{cases}
			\nabla \cdot (\abs{\nabla u}^{p-2} \nabla u ) = 0 &\text{on } B\setminus\Omega,\\
			\abs{\nabla u}^{p-2} \nabla u \cdot \hat{n} = \sigma \cdot \hat{n} & \text{on } \partial\Omega,\\
			u = 0 &\text{on }\partial B.
		\end{cases}
	\end{equation*}
	We claim that
	\begin{equation*}
		\overline{\sigma} :=
		\begin{cases}
			\sigma &\text{if } x \in \Omega\\
			\abs{\nabla u}^{p-2} \nabla u  &\text{if } x \in B\setminus\Omega
		\end{cases}
		\qquad\text{and}\qquad
		\overline{\rho} :=
		\begin{cases}
			\rho &\text{if } x \in \Omega\\
			0  &\text{if } x \in B\setminus\Omega
		\end{cases}
	\end{equation*}
	extend $\sigma$ and $\rho$ to functions in $L^q(B;\mathbb{R}^{m\times n})$ and $L^q(B; \mathbb{R}^m)$, respectively, and satisfy $\nabla \cdot \overline{\sigma} = \overline{\rho}$ in the sense of distribution.
	Indeed, it is clear that $\overline{\rho} \in L^q(B; \mathbb{R}^m)$, while $\overline{\sigma} \in L^q(B;\mathbb{R}^{m\times n})$ because $q=p/(p-1)$ and, consequently,
	\begin{equation*}
		\int_{B} \abs{\overline{\sigma}}^q \dVolume
		= \int_\Omega \abs{\sigma}^q \dVolume
		+ \int_{B\setminus\Omega} \abs{\nabla u}^{(p-1)q} \dVolume
		= \int_\Omega \abs{\sigma}^q \dVolume
		+ \int_{B\setminus\Omega} \abs{\nabla u}^{p} \dVolume
		<\infty.
	\end{equation*}
	The identity $\nabla \cdot \overline{\sigma} = \overline{\rho}$ is verified by a direct calculation: given any smooth and compactly supported test function $\varphi: B \to \mathbb{R}^m$,
	\begin{multline*}
		\int_B (\nabla \cdot \overline{\sigma})\cdot \varphi \dVolume
		= -\int_B \overline{\sigma} \cdot \nabla \varphi \dVolume
		= -\int_\Omega \sigma \cdot \nabla \varphi \dVolume
		   -\int_{B\setminus\Omega} \abs{\nabla u}^{p-2}\nabla u \cdot \nabla \varphi \dVolume
		\\
		=
		\int_\Omega \underbrace{(\nabla \cdot \sigma)}_{=\rho} \cdot \varphi \dVolume
		+ \int_{B\setminus\Omega} \underbrace{[\nabla \cdot (\abs{\nabla u}^{p-2}\nabla u)]}_{=0} \cdot \varphi \dVolume
		= \int_B \overline{\rho} \cdot \varphi \dVolume,
	\end{multline*}
	where in the second line we used the boundary conditions on $u$.
	Standard mollification arguments then yield a sequence $\{{\sigma}_k,{\rho}_k\}$ of functions that are smooth on $B$ (hence, on $\overline{\Omega}$), converge to $\overline{\sigma}$ and $\overline{\rho}$ in $L^q$, and satisfy $\nabla \cdot {\sigma}_k = {\rho}_k$.
\end{proof}

\subsection{Proof of \texorpdfstring{\cref{th:sharpness}}{Theorem \ref{th:sharpness}}}
\label{ss:sharpness-proof}
Since $\optimalValue \geq \occBound \geq \pdrBound$ by \cref{th:weak-duality}, it suffices to show that $\pdrBound \geq \optimalValue - \varepsilon$ for every $\varepsilon>0$.
To do so, we use an almost-optimal solution of the dual problem \cref{e:conjugate-dual-problem} obtained in \cref{lemma:sharpness} to construct  $(\phi_\epsilon,\eta_\epsilon, h_\epsilon, \ell_\epsilon) \in \mathbb{B}$ feasible for the maximization problem defining $\pdrBound$ in~\cref{e:pdr-bound} and satisfying
\begin{equation}\label{e:epsilon-suboptimality}
    \int_\Omega h_\epsilon(x) \, \dVolume + \int_{\partial\Omega} \ell_\varepsilon(x) \, \dSurf \geq \optimalValue - \varepsilon.
\end{equation}

Let $\sigma_\epsilon \in L^q(\Omega;\R^{m\times n})$ and $\rho_\varepsilon \in L^q(\Omega;\R^{m})$ satisfy $\nabla \cdot \sigma_\varepsilon = \rho_\varepsilon$, as well as
\begin{equation*}
    \int_\Omega -f_0^*(x, \sigma_\epsilon) - f_1^*(x,\rho_\epsilon) \, \dVolume \geq \optimalValue-\varepsilon.
\end{equation*}
The existence of such functions is guaranteed by \cref{lemma:sharpness}.
Using \cref{lemma:continuity,lemma:smoothness}, we may assume $\sigma_\varepsilon \in C^\infty(\overline{\Omega}; \R^{m\times n})$ and $\rho_\varepsilon \in C^\infty(\overline{\Omega}; \R^m)$.
Set
\begin{align*}
	\varphi_\epsilon(x,y) &= -\sigma_\epsilon(x) y, &
	\eta_\epsilon&=0,&
	h_\epsilon(x) &= -f_0^*(x, \sigma_\epsilon(x)) - f_1^*(x,\rho_\epsilon(x)),&
	\ell_\epsilon=0.
\end{align*}
Note that $\phi_\varepsilon \in \phiSpace{p}$ since $\sigma_\varepsilon$ is smooth up to the boundary of $\Omega$, and that $h_\varepsilon \in C(\overline{\Omega})$ because $f_0^*$ and $f_1^*$ are continuous (cf. the proof of \cref{lemma:continuity}).
These choices satisfy \cref{e:epsilon-suboptimality}, so to prove \cref{th:sharpness} we only need to check the inequalities defining the set $\mathbb{B}$ in~\cref{e:B-set}.

Given the constraints of~\cref{e:convex-integrands}, these inequalities become
\begin{align*}
    F^{\phi_\varepsilon,\eta_\varepsilon,h_\varepsilon} &\geq 0 \quad \text{on }\Gamma = \overline{\Omega} \times \R^m \times \R^{m \times n},\\
    G^{\phi_\varepsilon,\eta_\varepsilon,\ell_\varepsilon} &\geq 0 \quad \text{on }\Lambda =  \partial\Omega \times \{0\}.
\end{align*}
The latter is satisfied because problem~\cref{e:convex-integrands} has no boundary terms ($g=0$) and $\phi_\epsilon(x,y)=0$ on $\Lambda$, so $G^{\phi_\varepsilon,\eta_\varepsilon,\ell_\varepsilon}$ vanishes identically on that set. For the first inequality, instead, note that $\D\phi_\varepsilon = - (\nabla \cdot \sigma_\epsilon) \cdot y - \sigma_\varepsilon \cdot z = - \rho_\epsilon \cdot y - \sigma_\varepsilon \cdot z$ and, therefore,
\begin{align*}
    F^{\phi_\varepsilon,\eta_\varepsilon,h_\varepsilon}(x,y,z)
    = &\left[ f_0(x,z) - \sigma_\varepsilon(x) \cdot z + f_0^*\!\left(x, \sigma_\epsilon(x) \right) \right] \\
    &+ \left[ f_1(x,y) - \rho_\varepsilon(x) \cdot y + f_1^*\!\left(x,\rho_\varepsilon(x) \right) \right].
\end{align*}
By definition of the Legendre transforms $f_0^*$ and $f_1^*$, each square bracket is nonnegative for every $x \in \Omega$, so
$F^{\phi_\varepsilon,\eta_\varepsilon,h_\varepsilon} \geq 0$
 on $\Gamma$ as required.
\Cref{th:sharpness} is proved.

\section{Characterizing minimizers and near-minimizers}\label{s:optimizers}

If $\optimalValue=\pdrBound$, so that the occupation measure and pointwise dual relaxation bounds are sharp, solutions of the pointwise dual relaxation~\cref{e:pdr-bound} can be used to derive necessary optimality conditions for global minimizers of the original  problem~\cref{e:inf-problem-general}. We explain why in this brief section, before turning to our examples of non-sharpness.

\begin{proposition}\label{th:sm:optimizers}
	Suppose that $\mathcal{F}^* = \pdrBound$. If $u$ solves the minimization problem~\cref{e:inf-problem-general} and $(\phi,\eta,h,\ell) \in \mathbb{B}$ solves its pointwise dual relaxation~\cref{e:pdr-bound},
	\begin{align*}
	F^{\phi, \eta, h}(x,u(x),\nabla u(x))&=0 \quad \text{ a.e. on } \Omega,\\
	G^{\phi, \eta, \ell}(x,u(x)) &=0 \quad \text{ a.e. on } \partial\Omega.
	\end{align*}
\end{proposition}
\begin{remark}
    These optimality conditions distinguish local minimizers from global ones, unlike Euler--Lagrange equations (with the exception of convex problems).
\end{remark}

\begin{remark}
	Recall from \cref{th:sharpness-corollary} that if $n=1$ or $m=1$, then any nonconvex dependence on $\nabla u$ can be convexified while preserving the minimum. In this case, optimality conditions alternative to those in \cref{th:sm:optimizers} take the form of a differential inclusion; see \cite{Munoz2000} for an example. Elucidating the precise relation between the various optimality conditions in the literature is beyond the scope of this work.
\end{remark}

\begin{proof}[Proof of~\cref{th:sm:optimizers}]
	Let
	\begin{equation}\label{e:optimal-conditions-Psi}
		\Psi(u) = \int_\Omega f(x,u,\nabla u) \dVolume + \int_{\partial\Omega} g(x,u) \dSurf
	\end{equation}
	be the integral functional being minimized in~\cref{e:inf-problem-general}, and suppose that $u$ is optimal. Since $(\phi,\eta,h,\ell)$ attains equality in~\cref{e:pdr-bound} and $\pdrBound = \optimalValue=\Psi(u)$, we have
	\begin{equation}\label{e:exact-bound}
	\Psi(u) = \int_{\Omega} h(x) \dVolume + \int_{\partial\Omega} \ell(x) \dSurf.
	\end{equation}
	Next, recall the definitions of the sets $\Gamma$ and $\Lambda$ from~\cref{e:Gamma-def,e:Lambda-def}. Since $(\phi,\eta,h,\ell)$ is feasible for the dual relaxation~\cref{e:pdr-bound}, the functions $F^{\phi, \eta, h}$ and $G^{\phi, \eta, \ell}$ are nonnegative on $\Gamma$ and $\Lambda$, respectively. Moreover, $(x,u(x),\nabla u(x)) \in \Gamma$ for almost every $x \in \Omega$ and
	$(x,u(x)) \in \Lambda$
	for almost every $x  \in \partial\Omega$ since $u$ satisfies the constraints in~\cref{e:constraints}. Thus,
	\begin{eqnarray*}
	\Psi(u)
	&\leq& \Psi(u)
		  + \int_{\Omega} F^{\phi,\eta,h}(x,u,\nabla u)  \dVolume
		  + \int_{\partial\Omega} G^{\phi,\eta,\ell}(x,u)  \dSurf
	\\
	&\overset{\text{\cref{e:exact-bound}}}{=} &
	\int_{\Omega}\![F^{\phi,\eta,h}(x,u,\nabla u) + h(x)]  \dVolume + \int_{\partial\Omega}\![G^{\phi,\eta,\ell}(x,u) + \ell(x)]  \dSurf
	\\
	&\overset{\text{\cref{e:constraints-general-1}\&\cref{e:div-theorem}}}{=}&
	\int_{\Omega} f(x,u,\nabla u) \dVolume + \int_{\partial\Omega} g(x,u)\dSurf
	\\
	&=& \Psi(u).
	\end{eqnarray*}
	The first inequality is therefore an equality, so the functions $F^{\phi, \eta, h}(x,u(x),\nabla u(x))$ and $G^{\phi, \eta, \ell}(x,u(x))$ vanish almost everywhere on $\Omega$ and $\partial\Omega$, respectively.
\end{proof}

The previous result assumed that the pointwise dual relaxation can be solved, i.e., that it admits an optimizer. In lieu of this, a partial characterization of `near optimizers' can still be obtained provided that $\optimalValue=\pdrBound$.
Let $\Psi$ be as in~\cref{e:optimal-conditions-Psi}.

\begin{proposition}
	Let $\optimalValue = \pdrBound$.
	Given $\varepsilon>0$, let $u$ satisfy $\Psi(u)\leq \optimalValue + \varepsilon$ and let $(\phi,\eta,h,\ell) \in \mathbb{B}$ be admissible for the pointwise dual relaxation~\cref{e:pdr-bound} with $\optimalValue-\varepsilon \leq \int_{\Omega} h(x) \dVolume + \int_{\partial\Omega} \ell(x)$.
	Then, for every $\delta>0$, the volume and surface measures
	\begin{align*}
	\lambda_\delta &:= \abs{ \{x \in \Omega:\; F^{\phi \eta h}(x,u(x),\nabla u(x))\geq \delta\} }\\
	\sigma_\delta &:= \abs{ \{x \in \partial\Omega:\; G^{\phi \xi \ell}(x,u(x))\geq \delta\} }
	\end{align*}
	satisfy
	\[
	\lambda_\delta + \sigma_\delta \leq \frac{2\varepsilon}{\delta}.
	\]
\end{proposition}

\begin{proof}
	Our assumptions on $u$ and $(\phi,\eta,h,\ell)$ imply that
	\begin{align*}
	2\varepsilon
	&\geq
	\int_{\Omega} \left[ f(x,u,\nabla u) + h(x)\right] \dVolume
	+ \int_{\partial\Omega} \left[ g(x,u) + \ell(x)\right] \dSurf
	\\
	&=\int_{\Omega} F^{\phi,\eta,h}(x,u,\nabla u) \dVolume
	+ \int_{\partial\Omega} G^{\phi,\eta,\ell}(x,u) \dSurf,
	\end{align*}
	where the second line follows because the terms added to the integrands give a zero net contribution by~\cref{e:constraints-general-1,e:div-theorem}.
	Since the tuple $(\phi,\eta,h,\ell)$ is feasible for~\cref{e:pdr-bound}, the functions $x \mapsto F^{\phi \eta h}(x,u(x),\nabla u(x))$ and $x \mapsto G^{\phi \xi \ell}(x,u(x)) $ are nonnegative. The result follows from Chebyshev's inequality.
\end{proof}

\section{Non-sharpness of occupation measure bounds}\label{s:failure}
We end with two examples where the relaxation by occupation and boundary measures presented in \cref{s:om-relaxation} is not sharp, meaning that $\optimalValue > \occBound$. In both cases, we exhibit a pair $(\mu,\nu)$ of measures satisfying the constraints of the occupation measure relaxation~\cref{e:occ-measures-minimization} and for which the objective value is strictly less than $\optimalValue$. \Cref{ss:failure-double-well} treats a vectorial example which is nonconvex in $\nabla u$, coming from mathematical material science. A different example with convex dependence on $\nabla u$ and nonconvex dependence on $u$ appears in \cite{Korda2022}.
\Cref{ss:failure-poincare} gives an example showing how occupation measures fail to capture optimal Poincar\'e constants for mean-zero functions in one-dimension. Other examples of non-sharpness based on non-convex constraints also appear in  \cite{Korda2022}.

\subsection{A double-well problem}\label{ss:failure-double-well} Fix any $n,m\geq 2$, let $\Omega \subset \R^n$ be a bounded Lipschitz domain, and consider the non-convex minimization problem
\begin{equation}\label{e:double-well-problem}
\optimalValue := \inf_{\substack{u:\Omega \to \R^m \\[0.5ex] u = 0 \text{ on } \partial\Omega}} \int_{\Omega} \abs{\nabla u(x) - A}^2 \, \abs{\nabla u(x) -B}^2\, \dVolume,
\end{equation}
where $\abs{\cdot}$ is the Frobenius matrix norm. We take $A,B\in\mathbb{R}^{m\times n}$ to satisfy
\begin{equation}\label{e:double-well:rank-condition}
    \rank(A-B)\geq 2
\end{equation}
which is possible given our choices for $n$ and $m$.
It is a well-known but non-trivial fact that \cref{e:double-well-problem} does not have a minimizer, even though its infimum  $\optimalValue$ is finite.  To evaluate it, then, one must find a way to pass to the limit along minimizing sequences that drive the integral to its infimal value. Here, we explain this using the theory of gradient Young measures, which we briefly recall; see~\cite{Muller1999,Pedregal1997,Rindler2018} for details.

A $W^{1,p}$-\emph{gradient Young measure} generated by a weakly converging sequence $\{u_k\}$ $\subset W^{1,p}(\Omega;\R^m)$ is a family $\{\gym\}_{x \in \Omega}$ of probability measures on $\R^{m\times n}$ such that~\cite{Rindler2018}:
\begin{enumerate}[(i),leftmargin=!,align=left,labelwidth=\widthof{(iii)},topsep=0.75ex,itemsep=0.5ex]
    \item $\int_\Omega \int_{\R^{m\times n}} \abs{z}^p \dgym(z) \dVolume < \infty$;
    \item\label{gym:measurability} Given any $\varphi:\Omega \times \mathbb{R}^{m\times n} \to \mathbb{R}$ that is measurable in $x$ and continuous in $z$,\footnote{Such functions are called \emph{Carath\'eodory functions}.} the function $x \mapsto \int \varphi(x,z) \dgym(z)$ is measurable;
    \item Given any $\varphi$ as in \ref{gym:measurability} such that $\{\varphi(x,\nabla u_k)\}$ is uniformly $L^1$-bounded and equiintegrable, $\varphi(x,\nabla u_k)$ converges to $\int_{\mathbb{R}^{2\times 2}}\varphi(x,z) {\rm d}\gym(z)$ weakly in $L^1(\Omega)$.
\end{enumerate}
The space of all $W^{1,p}$-gradient Young measures is denoted by $\GYMspace{p}(\Omega;\R^{m\times n})$.

Applying the theory of gradient Young measures to problem \cref{e:double-well-problem}, one finds that
\begin{equation}\label{e:gym-relaxation}
\optimalValue = \min_{
			\substack{
        		u \in W^{1,p}(\Omega;\R^{m\times n})\\
				\{\gym\} \in \GYMspace{p}(\Omega;\R^{m\times n})\\
				\int z \dgym(z) = \nabla u(x) \\
				u = 0 \text{ on } \partial\Omega
				}}
			\int_{\Omega}\int_{\R^{2\times 2}} \abs{z - A}^2 \abs{z -B}^2 \, \dgym(z) \dVolume
\end{equation}
(see, e.g., \cite[Theorem~4.9]{Muller1999}). In contrast, the occupation measure relaxation is
\begin{equation}\label{e:double-well-om-relax}
\occBound = \min_{
    (\mu,\nu) \in \omSpace{p}
	}
	\langle \abs{z - A}^2 \abs{z - B}^2, \mu \rangle,
\end{equation}
where for our example the set $\omSpace{p}$ from~\cref{e:A-set-definition} is
\begin{equation*}
    \omSpace{p} = \left\{
        (\mu,\nu) \in \measSpace:\; \text{\cref{e:marginal-constraints-mu}, \cref{e:marginal-constraints-nu}
        and \cref{e:divergence-constraints-measures}}
    \right\}.
\end{equation*}
Whether or not the minima in \cref{e:gym-relaxation} and \cref{e:double-well-om-relax} are the same comes down to the relation between gradient Young and occupation measures.

The pushforward operation provides a natural embedding of $W^{1,p}\times \GYMspace{p}$ into $\omSpace{p}$. Precisely, to each Sobolev function--gradient Young measure pair $(u,\{\gym\})$ we can associate the occupation and boundary measures
\begin{align*}
	\mu({\rm d}x, {\rm d}y, {\rm d}z) &:= \dVolume \otimes \delta_{u(x)}({\rm d}y) \otimes \gym( {\rm d} z),
	\\
	\nu({\rm d}x, {\rm d}y) &:= \dSurf \otimes \delta_{u(x)}({\rm d} y).
\end{align*}
However, this embedding is far from a one-to-one correspondence since when $n,m\geq 2$ the definition of $\omSpace{p}$ misses crucial nonlinear constraints satisfied by gradient Young measures. Indeed, the latter satisfy the nonlinear Jensen-type inequalities~\cite{Kinderlehrer1994}
\begin{equation}\label{e:quasiconvex-jensen}
h\!\left(\int_{\mathbb{R}^{m\times n}} z \dgym(z) \right)
\leq
\int_{\R^{m\times n}} h(z) \dgym(z)
\end{equation}
for every function $h: \R^{m \times n} \to \R$ that grows no faster than $|z|^p$ and is \emph{quasiconvex}, i.e., $h(z) \leq \fint_D h(z+\nabla \phi) \dVolume$ for all $z\in\R^{m\times n}$ and all compactly supported smooth functions $\phi:D \to \R^m$ ($D \subset \R^n$ can be any bounded Lipschitz domain~\cite[Ch.~5]{Rindler2018}). The inability of occupation measure relaxations to capture these nonlinear constraints results in a relaxation gap.

\begin{proposition}\label{th:relaxation-gap-double-well}
	For the minimization problem in~\cref{e:double-well-problem}, $\optimalValue>0$ but $\occBound = 0$.
\end{proposition}

\begin{proof}
    It suffices to consider the case $A=I$, $B=-I$ by a change of variables. We also fix $n=m=2$, since the general case is analogous. The exponent is $p=4$.

	First, we prove that $\optimalValue > 0$. By contradiction, if the minimimum of the gradient Young measure problem~\cref{e:gym-relaxation} is zero, it is solved by a pair $(u,\{\gym\})$ with
	\begin{equation}\label{e:double-well-opt-gym}
		\gym(z) = \theta(x) \delta_I(z) + \left( 1-\theta(x) \right) \delta_{-I}(z)
	\end{equation}
	for $\theta: \Omega \to [0,1]$. However, this is incompatible with the constraint $\int z\dgym = \nabla u(x)$ and the inequalities~\cref{e:quasiconvex-jensen}. Specifically, since $\det(z)$ and $-\det(z)$ are both quasiconvex (see, e.g., \cite[Corollary~5.9]{Rindler2018}), any gradient Young measure in~\cref{e:gym-relaxation} must obey
	\begin{equation*}
		\det\left(\nabla u(x)\right) = \int_{\R^{2\times 2}} \det(z) \dgym(z) \qquad \text{a.e. } x \in \Omega.
	\end{equation*}
	This identity fails for the measure in~\cref{e:double-well-opt-gym} because
	\begin{equation*}
		\int_{\R^{2\times 2}} \det(z) \dgym(z)
		= \theta(x)\det(I) + \left( 1-\theta(x) \right) \det(-I)
		= 1,
	\end{equation*}
	while
	\begin{equation*}
		\int_\Omega \det(\nabla u) \dVolume
		= \int_\Omega \nabla \cdot\left(u_1 \frac{\partial u_2}{\partial x_2}, -u_1 \frac{\partial u_2}{\partial x_1} \right) \dVolume
		= \int_{\partial\Omega} \left(u_1 \frac{\partial u_2}{\partial x_2}, -u_1 \frac{\partial u_2}{\partial x_1} \right) \cdot \hat{n} \dSurf
		= 0
	\end{equation*}
by the boundary condition $u=0$ at $\partial\Omega$.

	Next, we show that $\occBound = 0$. It is clear that $\occBound \geq 0$, so we only need to find measures $\mu$ and $\nu$ that are feasible for~\cref{e:double-well-om-relax} and such that $\langle \abs{z - I}^2 \abs{z + I}^2, \mu \rangle = 0$. We claim that suitable choices are
	\begin{gather*}
	\mu
	:=
	\dVolume \otimes \delta_{0}({\rm d}y) \otimes \left( \frac12 \delta_{I}({\rm d}z) + \frac12 \delta_{-I}({\rm d}z)\right)
	\\
	\nu
	:=
	\dSurf \otimes \delta_{0}({\rm d}y).
	\end{gather*}
	It is clear that $\mu$ and $\nu$ have finite moments of order $p$ or less. It is also clear that the $x$-marginals of $\mu$ and $\nu$ coincide with the volume and surface measures on $\Omega$, as required by \cref{e:marginal-constraints-mu,e:marginal-constraints-nu}.
	To check~\cref{e:divergence-constraints-measures}, fix any $\varphi \in \phiSpace{4}$ and, summing over repeated indices to simplify the notation, observe that
	\begin{multline*}
	\langle \D\varphi, \mu \rangle
	= \int_\Omega \left[
		\frac{\partial\varphi_i}{\partial x_i} + \frac{\partial\varphi_i}{\partial x_j} \left(\frac12 I_{ji} - \frac12 I_{ji} \right)\right]_{(x,0)} \dVolume
	\\
	=\int_\Omega \nabla \cdot \varphi(x,0) \dVolume
	=\int_{\partial\Omega} \varphi(x,0) \cdot \hat{n}(x) \dSurf
	= \langle \varphi \cdot \hat{n}, \nu \rangle.
	\end{multline*}
	An analogous calculation shows that $\langle \abs{z - I}^2 \abs{z + I}^2, \mu \rangle = 0$, concluding the proof of \cref{th:relaxation-gap-double-well}.
\end{proof}

\subsection{Optimal Poincar\'e constant for univariate mean-zero functions}
\label{ss:failure-poincare}

For our next example, we consider a one-dimensional minimization problem giving the optimal Poincar\'e constant for univariate mean-zero functions on $\Omega = (-1,1)$,
\begin{equation}\label{e:poincare-mean-zero}
\mathcal{F}^* := \min_{\substack{\int_{-1}^1 u^2 \dVolume = 1\\ \int_{-1}^1 u \dVolume = 0}} \;\int_{-1}^1 \abs{\frac{{\rm d} u}{{\rm d}x}}^2 \dVolume.
\end{equation}
The occupation measure relaxation for this minimization problem reads
\begin{equation}\label{e:poincare-mean-zero-om-relax}
\occBound = \min_{
    (\mu,\nu) \in \omSpace{2}
}
\langle z^2, \mu \rangle,
\end{equation}
where the constraints set $\omSpace{2}$ in \cref{e:A-set-definition} now includes the constraints $\langle y, \mu \rangle = 0$ and $\langle y^2, \mu \rangle = 1$ as well as \cref{e:marginal-constraints-mu},\cref{e:marginal-constraints-nu} and \cref{e:divergence-constraints-measures}.

The next result proves that $\optimalValue>\occBound$. The relaxation gap arises because the \textit{average} of the occupation and boundary measures generated by the constant functions $u(x)=\pm\smash{1/\sqrt{2}}$ is feasible for~\cref{e:poincare-mean-zero-om-relax}, even though their generating functions are clearly not admissible in~\cref{e:poincare-mean-zero} (they violate the mean-zero constraint).

\begin{proposition}\label{th:relaxation-gap-poincare-mean-zero}
	For the minimization problem in~\cref{e:poincare-mean-zero}, $\optimalValue=\smash{\frac{\pi^2}{4}}$ but $\occBound = 0$.
\end{proposition}

\begin{remark}
    The reader may want to contrast this with \cite{Chernyavskiy2020}, which proves sharpness for some optimal Poincar\'e constants in the Dirichlet boundary condition case.
\end{remark}
\begin{proof}
	The optimal value $\mathcal{F}^*= \smash{\frac14}\pi^2$ is easily computed upon solving the Euler--Lagrange equation for problem~\cref{e:poincare-mean-zero}, and it is attained by $u^*(x) = \sin(\pi x/2)$.

	To prove that $\occBound=0$ we observe that $\occBound$ is clearly nonnegative, so it suffices to exhibit a pair $(\mu,\nu)$ that is feasible for the relaxed problem~\cref{e:poincare-mean-zero-om-relax} and achieves $\langle z^2, \mu \rangle = 0$. For this, we consider
	\begin{gather*}
	\mu
	:= \tfrac12 \dVolume \otimes
		\left[\delta_{\frac{1}{\sqrt{2}}}({\rm d}y) + \delta_{-\frac{1}{\sqrt{2}}}({\rm d}y)\right] \otimes
		\delta_{0}({\rm d}z),
	\\
	\nu
	:= \tfrac12  \Big[\delta_{-1}({\rm d}x) + \delta_{1}({\rm d}x)\Big] \otimes
	\left[\delta_{\frac{1}{\sqrt{2}}}({\rm d}y) + \delta_{-\frac{1}{\sqrt{2}}}({\rm d}y)\right]
	\end{gather*}
	These measures clearly have bounded moments of order up to $2$, and their $x$-marginals are the volume and surface measures by \cref{e:marginal-constraints-mu,e:marginal-constraints-nu}. Further, for any $\phi \in \phiSpace{2}$
	\begin{align*}
	\langle \mathcal{D}\phi, \mu \rangle
	&= \frac12  \int_{-1}^{1} \frac{\partial\varphi}{\partial x} \!\left(x,\tfrac{1}{\sqrt{2}}\right) + \frac{\partial\varphi}{\partial x}\!\left(x,-\tfrac{1}{\sqrt{2}}\right) \;\dVolume
	\\
	&= \frac12 \left[ \phi\!\left(1,\tfrac{1}{\sqrt{2}}\right) + \phi\!\left(1,-\tfrac{1}{\sqrt{2}}\right)
	- \phi\!\left(-1,\tfrac{1}{\sqrt{2}}\right) - \phi\!\left(-1,-\tfrac{1}{\sqrt{2}}\right)
	\right]
	\\
	&=\langle \phi \cdot n, \nu \rangle.
	\end{align*}
	In the last step, we used our choice of boundary measure $\nu$ and the fact that the unit normal vector to the boundary of $\Omega=(-1,1)$ is $\hat{n}(x)=x$.
	Similar calculations show that  $\langle y, \mu \rangle = 0$,  $\langle y^2, \mu \rangle = 0$, and $\langle z^2, \mu \rangle = 0$. \Cref{th:relaxation-gap-poincare-mean-zero} is therefore proved.
\end{proof}

\section{Conclusions}\label{s:conclusions}
We have investigated the problem of finding \textit{a priori} lower bounds on integral minimization problems. We have shown that lower bounds obtained using the occupation measure framework of \cite{Awi2014,Korda2018} can be estimated from below by a dual problem, which coincides with the pointwise dual relaxation developed in \cite{Chernyavskiy2020} using ideas introduced in the context of Lyapunov methods for partial differential equations \cite{Ahmadi2019,Ahmadi2016,Ahmadi2017,Valmorbida2015b,Valmorbida2015c,Valmorbida2015a}.
We have also identified general sufficient conditions ensuring that this duality is strong (\cref{th:strong-duality-noncompact}), in which case the occupation measure and pointwise dual relaxations are the same. Checking our sufficient conditions for strong duality requires one to construct a function $\varphi_0$ satisfying the two coercivity inequalities \cref{e:strong-duality-conditions-coercivity-bulk,e:strong-duality-conditions-coercivity-boundary}.
In applications, these inequalities may simply hold with $\varphi_0=0$. Our choice to include non-zero $\phi_0$ allows for more general problems where coercivity is not evident, but can be proved after a translation.
If the constraints imposed on the original problems imply uniform bounds on admissible functions and their gradients, as assumed in \cite{Korda2018,Korda2022}, coercivity is immediate and our strong duality theorem applies.

In the second part of this work, we proved that relaxations via occupation measures are sharp for variational problems with a particular additive convex structure. Our proof uses convex duality, which is an established route to study such variational problems (see, e.g., \cite{EkelandTemam1999}); here, we recognized its relation to the occupation measure framework. 
After this paper was submitted, different measure-theoretic arguments appeared in \cite{hkkz2023arxiv} extending our sharpness result to other convex problems including ones whose integrands do not necessarily have an additive structure.

On the other hand, as the results of this paper show, occupation measure relaxations are not always sharp: the counterexamples in \cref{s:failure} show that sharpness can fail even in the one-dimensional scalar case. The lack of sharpness for vectorial problems such as the double-well example of \cref{ss:failure-double-well} does not come as a surprise, because the linear constraints imposed by the occupation measure framework do not fully characterize the gradient Young measures needed for sharp relaxations \cite{Muller1999,Pedregal1997,Rindler2018}. This fundamental limitation of the method prevents occupation measures from producing sharp lower bounds in general. Moreover, the lack of a simple characterization of quasiconvexity makes it hard to incorporate the Jensen-type inequalities satisfied by gradient Young measures to improve occupation measure relaxations.

Nevertheless, the computational tractability of occupation measure relaxations \cite{Chernyavskiy2020,Korda2018} means that they should be added to the list of available techniques for proving nontrivial \emph{a priori} bounds on integral minimization problems, such as translation and calibration methods as well as polyconvexity (see, e.g., \cite{Buttazzo1998,Dacorogna2008,Firoozye1991,Milton2002composites}).
Future work should clarify the relationship between occupation measures and these well-known techniques.

\section*{Acknowledgments}
We thank D. Goluskin, J. Bramburger and A. Chernyavsky for many discussions about this work and for offering comments on an early draft. This work also benefited from conversations with A. Wynn, S. Chernyshenko, and G. Valmorbida. We gratefully acknowledge the hospitality and support of the Banff International Research Station; this paper was finished during the Focussed Research Group ``Studying PDE dynamics via optimization with integral inequality constraints'' (\url{http://www.birs.ca/events/2022/focussed-research-groups/22frg243}).

\section*{Open-access statement}
For the purpose of open access, the authors have applied a ‘Creative Commons Attribution' (CC BY) licence to any Author Accepted Manuscript version arising.

\section*{Data access statement}
No data were generated or analysed during this study.

\bibliographystyle{./siamplain}
\bibliography{reflist}

\end{document}